\documentclass{amsart}
\usepackage{amsmath,amsthm,amsfonts,amssymb,amscd,amsbsy}
\usepackage{caption,enumerate}
\usepackage{dsfont,lscape}
\usepackage{stmaryrd}
\usepackage[all]{xy}
\usepackage{hyperref}
\usepackage{cleveref}

\hypersetup{
    pdftoolbar=true,
    pdfmenubar=true,
    pdffitwindow=false,
    pdfstartview={FitH},
    pdftitle={},
    pdfauthor={},
    pdfsubject={},
    pdfkeywords={}
    pdfnewwindow=true,
    colorlinks=true, 
    linkcolor=blue,
    citecolor=blue,
    urlcolor=black,
}

\usepackage{mathtools}

\DeclarePairedDelimiter{\floor}{\lfloor}{\rfloor}

\newcommand{\N}{\mathds{N}}
\newcommand{\Z}{\mathds{Z}}
\newcommand{\R}{\mathds{R}}
\newcommand{\C}{\mathds{C}}

\newcommand{\Ca}{\mathds C\mathrm{a}}
\newcommand{\G}{\mathsf{G}}
\newcommand{\U}{\mathsf{U}}
\newcommand{\SO}{\mathsf{SO}}
\newcommand{\SL}{\mathsf{SL}}
\newcommand{\Spin}{\mathsf{Spin}}
\newcommand{\Cas}{\operatorname{Cas}}
\newcommand{\Sym}{\operatorname{Sym}}
\newcommand{\dd}{\mathrm d}
\newcommand{\g}{\mathrm g}
\newcommand{\diag}{\mathrm{diag}}

\newcommand{\sonc}{\mathfrak{so}(n,\C)}

\newcommand{\Ss}{\mathds{S}}
\newcommand{\Fr}{\mathrm{Fr}}
\newcommand{\Ric}{\operatorname{Ric}}
\newcommand{\scal}{\operatorname{scal}}

\allowdisplaybreaks

\newtheorem{theorem}{Theorem}[]
\newtheorem{lemma}[theorem]{Lemma}
\newtheorem{proposition}[theorem]{Proposition}
\newtheorem{corollary}[theorem]{Corollary}

\newtheorem{mainthm}{\sc Theorem}

\theoremstyle{definition}
\newtheorem{definition}[theorem]{Definition}

\theoremstyle{remark}
\newtheorem{remark}[theorem]{Remark}
\newtheorem{example}[theorem]{Example}

\allowdisplaybreaks
\numberwithin{equation}{section}

\allowdisplaybreaks

\title{Curvature operators and rational cobordism}


\author[R. G. Bettiol]{Renato G. Bettiol}
\address{\!\!\!\begin{tabular}{lll}
CUNY Lehman College & & CUNY Graduate Center \\
Department of Mathematics & & Department of Mathematics \\
250 Bedford Park~Blvd W & & 365 Fifth Avenue \\
Bronx, NY, 10468, USA & & New York, NY, 10016, USA
\end{tabular}
}
\email{r.bettiol@lehman.cuny.edu}

\author[M. J. Goodman]{McFeely Jackson Goodman}
\address{California Polytechnic State University\newline\indent
Department of Mathematics \newline\indent
San Luis Obispo, CA, 93407, USA}
\email{mgoodm06@calpoly.edu}

\allowdisplaybreaks
\numberwithin{equation}{section}
\numberwithin{theorem}{section}

\date{October 18, 2024}

\begin{document}
\begin{abstract}
We determine linear inequalities on the eigenvalues of curvature operators that imply vanishing of the twisted $\hat A$ genus on a closed Riemannian spin manifold, where the twisting bundle is any prescribed parallel bundle of tensors. 
These inequalities yield surgery-stable curvature conditions tailored to annihilate further rational cobordism invariants, such as the Witten genus, elliptic genus, signature, and even the rational cobordism class itself.
\end{abstract}

\maketitle

\section{Introduction}
Closed Riemannian spin manifolds $M$ with positive scalar curvature ($\scal>0$) have vanishing $\hat A$ genus $\hat A(M)=0$, by a celebrated theorem of Lichnerowicz~\cite{lichn}. Indeed, the Dirac operator $D$ on the spinor bundle of $M$ satisfies
\begin{equation}\label{eq:Lichn}
	D^2=\nabla^*\nabla + \tfrac{\scal}{4},
\end{equation}
and, by the Atiyah--Singer Index Theorem, $\hat A(M)\neq 0$ implies its kernel is nontrivial. Thus, $\hat A(M)\neq 0$ is a topological obstruction to the existence of Riemannian metrics with $\scal>0$ on such $M$; e.g., this shows that $K3$ surfaces do not admit $\scal>0$.

In this paper, we establish similar obstructions to stronger curvature conditions.
More precisely, we find curvature conditions that are, on the one hand, \emph{weak enough} to be satisfied by large classes of manifolds $M$, in particular, are preserved under surgeries of high codimension; 
on the other hand, \emph{strong enough} to imply vanishing of certain rational cobordism invariants if $M$ is spin, such as twisted $\hat A$ genera in Theorem~\ref{mainthm:A}, Witten genus, elliptic genus, and signature in Theorem~\ref{mainthm:witten}, or even all the Pontryagin numbers and hence the rational cobordism class itself in Theorem~\ref{mainthm:D}.

\subsection{\texorpdfstring{Curvature conditions $C_p(R)>0$}{The curvature conditions}}
Let \(R\colon \wedge^2TM\to\wedge^2TM\) be the curvature operator of $(M^n,\g)$, and \(\nu_1\leq \dots\leq \nu_{\binom{n}{2}}\) be its eigenvalues.
For each \(0<r\leq \binom{n}{2}\),~let
\[\Sigma(r,R)=\nu_1+\dots+\nu_{\floor{r}}+(r-\floor{r})\nu_{\floor{r}+1}.\] 
In particular, if $r\in\N$ is a positive integer, then $\Sigma(r,R)$ is the sum of the smallest $r$ eigenvalues of $R$, and $-\Sigma(r,-R)$ is the sum of the largest $r$ eigenvalues of $R$. Note that $2\,\Sigma\big({\binom{n}{2}},R\big)=\scal$ is the scalar curvature of $(M^n,\g)$.
For each \(p\in\N\), define
\begin{equation}\label{eq:apbp}
\textstyle r_p=\frac{n^2+(8p-1)n+8p(p-1)}{n+8p(p+1)} \quad \text{ and }\quad r'_p=\frac{n+p-2}{p}.	
\end{equation}
Let $\mu$ be the largest eigenvalue of $\Ric$, and define the functions $C_p(R)\colon M\to\R$ as
\begin{align*}
C_1(R)&=\textstyle\min\!\Big\{\!\!\left(\frac{n}{8}+2\right)\!\Sigma({r_1},R),\frac{\scal}{8}\!\Big\} +\frac{\scal}{8}-\mu, \\[2pt]
C_p(R)&=\textstyle \min\!\Big\{\!\!\left(\frac{n}{8}+p^2+p\right)\!\Sigma({r_p},R),\frac{n(n-1)}{8r_p}\Sigma({r_p},R) \!\Big\}+\frac{\scal}{8}+p^2\Sigma({r'_p},-R), \;\; p\geq2.
\end{align*}
Clearly, the above quantities are linear combinations of the eigenvalues of \(R\) if $p\geq 2$, and of $R$ and \(\Ric\) if \(p=1\). The coefficients of these linear combinations depend on the sign of \(\Sigma(r_p,R)\) due to the terms involving minima.
Moreover, \(C_{p}(R)>0\) implies \(C_{q}(R)>0\) for all \(1\leq q < p\), as well as \(\scal>0\), see Proposition~\ref{prop:nested}.

\subsection{\texorpdfstring{Twisted $\hat A$ genera}{Twisted A hat genera}}
Given any complex vector bundle $E\to M$, let
\begin{equation*}
\hat A(M,E)=\langle \hat A(TM)\cdot\operatorname{ch}(E),[M]\rangle,
\end{equation*}
where \(\hat{A}(TM)\) is the \(\hat{A}\) polynomial in the Pontryagin classes of \(TM\), and \(\operatorname{ch}(\cdot)\) is the Chern character. By the Atiyah--Singer Index Theorem, if $M$ is spin, then $\hat A(M,E)$ is equal to the index of the Dirac operator on spinors of $M$ twisted with~\(E\); in particular, it is an integer. Furthermore, if \(E\) is associated to the frame bundle of the Riemannian manifold \((M^n,\g)\) by a representation of \(\SO(n)\), then \(\hat{A}(M,E_\C)\)  is a universal rational linear combination of  Pontryagin numbers of \(M\) that depends only on the rational oriented cobordism class of \(M\), analogously to $\hat A(M)=\langle \hat A(TM),[M]\rangle$.

Our main result is the following family of vanishing theorems for $\hat A(M,E_\C)$:

\begin{mainthm}\label{mainthm:A}
Let \(M\) be a closed Riemannian spin manifold of dimension $n=4k$, $k\geq2$, and \(E\subseteq TM^{\otimes p}\) be a parallel subbundle. If \(C_p(R)>0\), then~$\hat A(M,E_\C)=0$.
\end{mainthm}	

For specific parallel subbundles \(E\subseteq TM^{\otimes p}\), e.g., \(E=\wedge^pTM\) or \(E=\Sym^pTM\), we give curvature conditions weaker than $C_p(R)>0$ that still  imply \(\hat{A}(M,E_\C)=0\), see Theorem~\ref{GeneralVanishing}. 
Simple examples of applications of Theorem~\ref{mainthm:A} are that $M=\mathds H P^2$ does not admit \(C_1(R)>0\), because it has $\hat A(M,TM_\C)\neq0$, see also Theorem~\ref{mainthm:D}; and $M=\Ca P^2$ does not admit \(C_2(R)>0\), since it has $\hat A(M,\wedge^2 TM_\C)\neq0$.

Let us examine in further detail the case in which $p=1$ and $E=TM$; this is of relevance in Mathematical Physics in connection with Rarita--Schwinger fields (spin $\frac32$ fermions), see \cite{homma-semmelmann}. 
Under symmetry assumptions, Dessai~\cite{dessai-pams} proved that spin manifolds with positive sectional curvature satisfy $\hat{A}(M) = \hat A(M, TM_\C)=0$.
Note that \(C_1(R)\geq 0\) if \(\Sigma({r_1},R)\geq0\), where $r_1=\tfrac{n(n+7)}{n+16}$, and $\frac{\scal}{8}\operatorname{Id}-\Ric\succeq0$; and \(C_1(R)>0\) if either term is positive.

\begin{mainthm}\label{A-props}
The curvature condition $C_1(R)>0$ satisfies the following:
\begin{enumerate}[\rm (i)]
\item it is preserved under surgeries of codimension at least \( 10\);
\item every oriented cobordism class \([M^{n}]\), \(n\geq 10,\) which is not a nontrivial torsion class is represented by a Riemannian manifold with \(C_1(R)>0;\)
\item every spin cobordism class \([M^{n}]\), \(n\geq 10,\) with $\hat{A}(M) = \hat A(M, TM_\C)=0$ has a multiple which is represented by a spin Riemannian manifold with \(C_1(R)>0.\)
\end{enumerate}		
\end{mainthm}	

By Theorem~\ref{A-props}, the spin condition in Theorem~\ref{mainthm:A} is necessary if \(p=1\). Indeed, without the spin condition, \(C_1(R)>0\) places no restriction on the rational cobordism type of a manifold of dimension at least 10, and the conclusion of Theorem~\ref{mainthm:A}, along with the vanishing of the \(\hat{A}\) genus, are the only restrictions on the rational spin cobordism type of a spin manifold such that \(C_1(R)>0.\) 
Furthermore, Theorem~\ref{A-props} implies that the curvature condition \(C_1(R)>0\) does not restrict any Betti numbers nor any individual Pontryagin numbers in sufficiently large dimensions. 

We also prove a surgery stability result for the curvature conditions \(C_p(R)>0\), $p\geq 2$. Namely, \(C_p(R)>0\) is preserved under surgeries of codimension \(d\) on manifolds of dimension \(n\) provided that \((d-1)(d-2)>8p(p+n-2)\), see Proposition~\ref{A-products}.

\subsection{Witten genus and elliptic genus}
The \emph{Witten genus} and the \emph{elliptic genus} are rational cobordism invariants that can be written as formal power series with coefficients given by linear combinations of Pontryagin numbers. The elliptic genus generalizes both \(\hat{A}\) and \(L\) genera, see \eqref{eq:phi-interp} for details. Using Theorem~\ref{mainthm:A}, we prove:

\begin{mainthm}\label{mainthm:witten}
	Let \(M\) be a closed Riemannian spin manifold of dimension $n=4k$. 
	\begin{enumerate}[\rm (i)]
		\item Set \(p=\lfloor{\frac{k}{6}}\rfloor-1\) if \(k\equiv 1\mod 6\), and \(p=\lfloor{\frac{k}{6}}\rfloor\) otherwise.  If \(p\geq1\), \(C_p(R)>0,\) and the first Pontryagin class of \(M\) vanishes, then the Witten genus of \(M\) vanishes.
		\item If \(k\geq 2\) and \(C_{\lfloor{k/2}\rfloor}(R)>0\), then the elliptic genus (and hence the signature) of \(M\) vanishes.  
	\end{enumerate}	
\end{mainthm}

The dimensional hypotheses in Theorem~\ref{mainthm:witten}, namely $n\geq 24$ and $n\neq 28$ in (i), and $n\geq 8$ in (ii), only exclude situations already covered by Lichnerowicz.  Namely, in (i), if \(p=0\) and the first Pontryagin class \(p_1(TM)\) vanishes, then the Witten genus vanishes if and only if the \(\hat{A}\) genus vanishes; in (ii), the Witten, elliptic, and \(\hat{A}\) genera in dimension $4$ all vanish if and only if the signature vanishes. In Section~\ref{section: ex}, we show that there are nontrivial rational cobordism classes containing manifolds satisfying the curvature conditions in Theorem~\ref{mainthm:witten} (i). Furthermore, the curvature conditions in (i) and (ii) are satisfied by round spheres, and those in (i) are stable under connected sums and other high-codimension surgeries (see Proposition~\ref{A-products}).

The Witten genus ought to vanish on closed Riemannian spin manifolds $M$ with $\Ric\succ0$ for which the spin characteristic class \(\frac12p_1(TM)\) vanishes (these are so-called \emph{string} manifolds, a condition weaker than \(p_1(TM)=0\)), according to a compelling conjecture of Stolz~\cite{stolz}. However, \(C_p(R)>0\) for \(p\) as in Theorem~\ref{mainthm:witten} is not implied by $\Ric\succ0$. Even more, as \(C_p(R)>0\) is stable under certain surgeries, we can produce examples of string manifolds with infinite fundamental group satisfying \(C_p(R)>0\), which therefore do not admit metrics with \(\Ric\succ0\) by Myer's theorem.
On the other hand, in dimensions $24\leq n < 48$ or $n=52$, the Witten genus of spin manifolds $M^n$ with $p_1(TM)=0$ vanishes if and only if a multiple of the spin cobordism class $[M]$ is represented by a spin Riemannian manifold with $C_1(R)>0$.

\subsection{Rational cobordism class}
Closed Riemannian manifolds (not necessarily spin) of dimension $n$ with $\Sigma(n-p,R)>0$ for some $0< p \leq \floor{\frac{n}{2}}$ have vanishing Betti number $b_p(M)=0$, by a recent work of Petersen and Wink~\cite{petersen-wink}. Combining this with Theorem~\ref{mainthm:A}, we find sufficient conditions for \emph{all} Pontryagin numbers to vanish, which, in turn, implies the vanishing of the rational cobordism class, namely:

\begin{mainthm}\label{mainthm:D}
Let \(M\) be a closed Riemannian spin manifold of dimension $n=4k$.
	\begin{enumerate}[\rm (i)]
	\item If $k=2$, $\Sigma(5,R)>0$, and \(M\) is Einstein, then $M$ is  null-cobordant.
	\item If \(k\geq 6\) is even, $\Sigma(2k+4,R)>0$, and \(\frac{\scal}{8}\operatorname{Id}-\Ric\succeq0\), then \(M\) is rationally null-cobordant.
	\item If \(k\geq9\) is odd, $\Sigma(2k+6,R)>0$, and \(\frac{\scal}{8}\operatorname{Id}-\Ric\succeq0\), then \(M\) is rationally null-cobordant. 
\end{enumerate}
\end{mainthm}

Closed oriented Riemannian manifolds of dimension $n=4k$ with $\Sigma(2k,R)>0$ are rational homology spheres by \cite{petersen-wink}; these are rationally null-cobordant. 

To facilitate comparison, let us further analyze the situation in dimension $n=8$, in which case a metric satisfies $\frac{\scal}{8}\operatorname{Id}-\Ric\succeq0$ if and only if it is Einstein, and rationally null-cobordant manifolds are null-cobordant as  there is no torsion in the cobordism group $\Omega_8^{\SO}$. A key example is the quaternionic projective plane $\mathds H P^2$, which is a spin manifold with signature $1$, thus not null-cobordant. While it follows from \cite{petersen-wink} that $\mathds H P^2$ does not support metrics with $\Sigma(4,R)>0$, we have by Theorem~\ref{mainthm:D} that $\mathds H P^2$ does not support Einstein metrics with $\Sigma(5,R)>0$; the same conclusions hold for connected sums $\#^\ell \mathds H P^2$, $\ell\geq0$. (The Fubini--Study metric on $\mathds H P^2$ is Einstein and has curvature operator $R\succeq0$ with kernel of dimension $18$, hence $\Sigma(r,R)>0$ only for $r\geq19$.) 

In dimensions $16$, $20$, and $28$, which are outside the scope of Theorem~\ref{mainthm:D}, there are stronger hypotheses that 
allow to reach the same conclusion, see~Theorem~\ref{thm:left-overs}.

\subsection{Key ideas and organization}
A sensible approach to seek generalizations of the Lichnerowicz obstruction to $\scal>0$  is to twist the Dirac operator with different vector bundles, causing changes in the curvature term in \eqref{eq:Lichn} that can be aimed at detecting other curvature conditions. 
More precisely, given a $\Spin(n)$-representation $\pi$, consider the twisted Dirac operator $D_\pi$ on the twisted spinor bundle $S\otimes E_\pi\to M$, where $S$ is the complex spinor bundle and $E_\pi$ is associated via $\pi$ to the principal bundle covering the frame bundle of $M$ by the spinor representation $\pi_S$. Then
\begin{equation*} 
    D_\pi^2=\nabla^*\nabla + \mathcal R_\pi,
 \end{equation*} 
where the endomorphism $\mathcal R_\pi$ is determined by the curvature operator $R$ of $M$ and the representation $\pi$; e.g., in case $\pi$ is the trivial representation one has $\mathcal R_\pi=\frac{\scal}{4}$, recovering \eqref{eq:Lichn}. The main difficulty in this approach is that $\mathcal R_\pi\succ0$ is algebraically very involved, and typically one can only ensure $\mathcal R_\pi\succ0$ by imposing unreasonably strong curvature assumptions such as $R\succ0$. Our first step towards overcoming this difficulty is to break $\mathcal R_\pi$ into simpler parts as explained in Lemma~\ref{HitchinDirac}, namely
\begin{equation*}
    \mathcal R_\pi=K(R,\pi_S\otimes \pi)+\tfrac{\scal}{8}\operatorname{Id}-K(R,\pi),
\end{equation*}
where $K(R,\pi)$ is the curvature endomorphism arising in the Weitzenb\"ock formula $\Delta_\pi =\nabla^*\nabla + t \, K(R,\pi) $ for Laplacians on the vector bundle $E_\pi$, up to a factor $t\in\R$, see \eqref{eq:krrho}; similarly for $K(R,\pi_S\otimes \pi)$ on the bundle $E_{\pi_S\otimes\pi}=S\otimes E_\pi$.
For instance, if $\pi$ is the defining representation of $\SO(n)$, then \(E_\pi=TM\) and \(K(R,\pi)=\Ric\).

The second step, motivated by recent works of Petersen and Wink~\cite{petersen-wink,petersen-wink-cplx,petersen-wink-3} and Nienhaus, Petersen, and Wink~\cite{nienhaus-petersen-wink}, is to prove:
 
 \begin{mainthm}\label{mainthm:pw}
 	 Let \(\pi\) be an irreducible orthogonal or unitary representation of \(\SO(n)\) or \(\Spin(n)\) with highest weight \(\lambda\). Let \(\rho\) be the half-sum of positive roots in \(\mathfrak{so}(n,\C)\) and \(r=\frac{\left<\lambda,\lambda+2\rho\right>}{\|\lambda\|^2}.\)  Then \(K(R,\pi)\succeq \|\lambda\|^2\,\Sigma(r,R)\operatorname{Id}\).
 \end{mainthm}

Theorem~\ref{mainthm:pw} generalizes the central estimate of \cite{petersen-wink}, which is also at the heart of the other papers cited above, and casts it in a natural representation-theoretic framework. We prove Theorem~\ref{mainthm:pw} in a more general formulation (Proposition~\ref{lowerbound}), relevant for manifolds with special holonomy.
This constitutes our main tool for relating the curvature endomorphisms $K(R,\pi)$ and $\mathcal R_\pi$ to linear combinations of the eigenvalues of $R$; with this in hand, standard Bochner-type arguments lead to the desired vanishing results, see Theorem~\ref{thm:bochner} for the case of~$\Delta_\pi$, and Theorem~\ref{GeneralVanishing} for that of~$D_\pi$. As explained in Section~\ref{sec:twistedDirac}, Theorem~\ref{mainthm:A} is a consequence of the latter, as $C_p(R)$ is defined so that Theorem~\ref{GeneralVanishing} can be applied to any subrepresentation $\pi$ of the $p^\text{th}$ tensor power of the defining representation. Regarding the former (see Section~\ref{sec:bochner}), besides recovering the results in \cite{petersen-wink,petersen-wink-cplx,petersen-wink-3} that relate $\Sigma(r,R)>0$ to the vanishing of Betti numbers in the Riemannian and K\"ahler settings, we are able to improve a Tachibana-type result relaxing the \emph{Einstein} assumption to \emph{harmonic curvature operator} (Theorem~\ref{thm:pwtachibana}), and to prove a vanishing theorem (Theorem~\ref{thm:killing}) for trace-free conformal Killing tensors on manifolds with $\Sigma(r,R)<0$, similar to results in \cite{DS10} and \cite{hms} for manifolds with $\sec<0$.

In Section~\ref{witten-genus}, we analyze the Witten genus and elliptic genus as modular forms to show  they vanish whenever the characteristic numbers described in Theorem~\ref{mainthm:A} vanish for sufficiently large \(p\), proving Theorem~\ref{mainthm:witten}.
Theorem~\ref{mainthm:D} is proven in Section~\ref{cobordism}, using the results of \cite{petersen-wink} to show that all but two Pontryagin numbers vanish, and then showing that \(\hat{A}(M)\) and $\hat A(M,TM_\C)$, which vanish by Lichnerowicz and Theorem~\ref{mainthm:A}, are linearly independent in those two remaining Pontryagin numbers. 
In Section~\ref{section: ex}, we examine examples of closed manifolds with \(C_1(R)>0\), including products of spheres, symmetric spaces, and Milnor hypersurfaces, proving Theorem~\ref{A-props}. We also compute \(C_p(R)\) for compact rank one symmetric spaces, examine surgery stability of \(C_p(R)>0\) using the criterion of \cite{hoelzel}, and identify nontrivial examples for which \(C_p(R)>0\) for the values of \(p\) given in Theorem~\ref{mainthm:witten}.

\subsection*{Acknowledgements} It is our great pleasure to thank Emilio Lauret and John Lott for assisting with representation theory questions and conversations about twisted Dirac operators, respectively. The first-named author is supported by the National Science Foundation, through grant DMS-1904342 and CAREER grant DMS-2142575. The second-named author is supported by the National Science Foundation, through grant DMS-2001985. We also thank the anonymous referee for the careful reading of our paper and thoughtful suggestions for improvements.

\section{Preliminaries}

In order to fix notation and conventions, we recall some basic facts from topology and representation theory; for details, see \cite{HBJ} and \cite{goodman-wallach}, respectively.

\subsection{Cobordisms, Pontryagin numbers, and genera}
Let \(n=4k\) and consider a multiindex \(I=(i_1, \dots ,i_\ell)\) of nonnegative integers such that \(\sum_{j=1}^\ell i_j=k.\) Given a closed oriented \(n\)-manifold \(M\), the Pontryagin number \(p_I(M)\) is defined as 
\[p_I(M)=\left<p_{i_1}(TM) \cdots p_{i_\ell}(TM),[M]\right>\in\Z,\]
where $p_{i}(TM)\in H^{4i}(M,\Z)$ are the Pontryagin classes of $TM\to M$.
If \(M\) and \(N\) are oriented cobordant, then \(p_I(M)=p_I(N)\), and \(p_I\) defines a map 
\[p_I\colon \Omega_n^\SO\longrightarrow\Z,\]
where \(\Omega^\SO_n\) is the oriented cobordism group.
If \(p(k)\) is the number of partitions of \(k\), and \(I_1, \dots ,I_{p(k)}\) is a list of those partitions, then, by the work of Thom  \cite{thom}, 
\begin{equation*}
(p_{I_1}, \dots ,p_{I_{p(k)}})\colon \Omega^\SO_n\otimes \mathds{Q}\longrightarrow\mathds{Q}^{p(k)}    
\end{equation*}
is an isomorphism.  Thus, if all Pontryagin numbers of an oriented manifold \(M\) vanish, then \(M\) is \emph{rationally null-cobordant}, that is, the disjoint union (or connected sum) of some number of copies of \(M\), all with the same orientation, bounds an oriented manifold.
It also follows from work of Thom that the natural map 
\begin{equation*} 
    \Omega^\Spin_n\otimes\mathds{Q}\longrightarrow\Omega^\SO_n\otimes\mathds{Q}
\end{equation*}
is an isomorphism, where \(\Omega^\Spin_n\) is the spin cobordism group. So, if \(M\) is spin and rationally null-cobordant, some number of copies of \(M\) bound a spin manifold as well. Accordingly, we shall refer to rationally null-cobordant manifolds without distinguishing between the spin and oriented cases. 

Let \(Q(x)=1+a_2x^2+a_4x^4+ \dots \) be an even formal power series. For variables \(x_1, \dots,x_m\), the product \(Q(x_1) \cdots Q(x_m)\) is a formal sum $K_0+K_1+K_2+ \dots$ of homogeneous symmetric polynomials \(K_i\) of degree \(i\) in the variables \(x_1^2, \dots ,x_m^2.\) Each \(K_i\) can be written as a polynomial \(K_i(\sigma_1, \dots ,\sigma_i)\) in the elementary symmetric polynomials \(\sigma_i=\sigma_i(x_1^2, \dots ,x_m^2).\) Note that \(K_i\) does not depend on \(m\) if \(i\leq m.\)  Then, we define 
\begin{align*}
K(TM)_i&=K_i\big(p_1(TM), \dots,p_i(TM)\big),\\
K(TM)&=1+K(TM)_1+K(TM)_2+ \dots
\end{align*}
For a closed orientable manifold \(M^{4k}\), set
\[K(M)=\left<K(TM)_{k},[M]\right>.\]
Thus, \(K(M)\) is a rational linear combination of Pontryagin numbers, and one checks that \(K(M\times N)=K(M)K(N),\) so $K$ defines a ring homomorphism
\[K\colon \Omega_*^{\SO}\otimes\mathds{Q}\longrightarrow\mathds{Q},\] which is called a \emph{(multiplicative) genus}. Let us mention the following examples:
\begin{enumerate}[\rm (i)]
	\item if \(Q(x)=\frac{x}{\tanh{x}}\), then the corresponding genus is called \(L\). The signature of an orientable manifold \(M\) is equal to \(L(M)\), by Hirzebruch's Signature Theorem;
	\item if \(Q(x)=\frac{x/2}{\sinh(x/2)}\), then the corresponding genus is called \(\hat{A}\). By the Atiyah--Singer Index Theorem, the index of the Dirac operator on a closed spin manifold $M$ is equal to \(\hat{A}(M)\);
	\item if \(Q(x)=1+x^{2k}\), then the corresponding genus is the \emph {Milnor invariant} \(s_{k}\).
\end{enumerate}

The latter is useful to describe the ring $\Omega^\SO_*$ due to theorems of Thom and Milnor:

\begin{theorem}\label{thom-milnor} Let \(M^{4k}\) be a closed oriented manifold for each positive integer \(k.\)  
	\begin{enumerate}[\rm (i)]
		\item \cite{thom} If \(s_{k}(M^{4k})\neq0\) for all \(k\), then \(\Omega_*^\SO\otimes\mathds Q=\mathds Q[M^4,M^8,\dots]\)
		\item {\cite[p.\ 207]{stong}} If
		\[s_{k}(M^{4k})=
		\begin{cases} \pm p,&\text{ if }2k+1=p^r,\ p\text{ a prime,}\\
		\pm 1,& \text{ otherwise,}\end{cases}\]
		then \(\Omega_*^\SO/\text{\normalfont torsion}=\Z[M^4,M^8,\dots].\)
		\end{enumerate}
\end{theorem}

\begin{example}\label{ex:generators-cobordism}
Useful generating sets can be constructed using Theorem~\ref{thom-milnor} and the computations  \(s_{k}(\C P^{2k})=2k+1\), and \(s_{k}(\mathds H P^{k})=2k+2-4^k,\) see \cite[Sec.~4.1]{HBJ} and \cite{gromov-lawson}. Furthermore, for each \(i,j\geq 2\), Milnor defined a \((2i+2j-2)\)-manifold \(H_{ij}\) which is a complex hypersurface of degree \((1,1)\) in \(\C P^i\times \C P^j\),  is the total space of a \(\C P^{j-1}\) bundle over \(\C P^{i},\) and for \(i+j\) odd satisfies \(s_{(i+j-1)/2}(H_{ij})=-\binom{i+j}{i}\), see \cite[p.\ 39]{HBJ}. One checks using divisibility properties of binomial coefficients that \(\Omega_*^\SO/\text{\normalfont torsion}\) is generated by \(\C P^{2k}\) and \(H_{ij}\).	
\end{example}

We describe the Chern character $\operatorname{ch}(E_\C)$ of the complexification $E_\C$ of a real vector bundle \(E\to M\) following a similar procedure. Consider the decomposition 
\[\textstyle\sum\limits_{i=1}^m2\cosh(x_i)=2m+\operatorname{ch}_2(\sigma_1)+\operatorname{ch}_4(\sigma_1,\sigma_2)+ \dots \]
into elementary symmetric polynomials \(\sigma_i=\sigma_i(x_1^2,\dots,x_m^2)\), and set 
\begin{align*}
\operatorname{ch}_{2i}(E_\C)&=\operatorname{ch}_{2i}\!\big(p_1(E), \dots,p_i(E)\big)\\
\operatorname{ch}(E_\C)&=\dim E+\operatorname{ch}_2(E)+\operatorname{ch}_4(E)+ \dots
\end{align*}
It follows that \(\operatorname{ch}(E\oplus F)=\operatorname{ch}(E)+\operatorname{ch}(F)\), and \(\operatorname{ch}(E\otimes F)=\operatorname{ch}(E)\operatorname{ch}(F)\).
Note 
\[\operatorname{ch}_{2i}(\sigma_1, \dots ,\sigma_i)=\frac2{(2i)!}\textstyle\sum\limits_{j=1}^mx_j^{2i},\]
so \(\operatorname{ch}_i(\cdot)\) can be computed using the relations between power sums and elementary symmetric polynomials given by Newton's identities.  

If \(E\) is associated to the frame bundle of \(M\) by a representation \(\pi\) of \(\SO(2m)\), let \(\{\pm\lambda_j\}\) be the set of weights of the complexification of \(\dd\pi\).  We consider each weight to be a linear functional \(\lambda_j\colon \C^m\to\C.\)  If we identify \(x_1,\dots,x_m\) with coordinates of \(\C^m\), then we can identify
\begin{equation}\label{eq:ch-weights}
\operatorname{ch}(E_\C)=\textstyle\sum\limits_j 2\cosh\!\big(\lambda_j(x_1,\dots,x_m)\big).
\end{equation}
Since the nonzero weights appear in pairs \(\pm\lambda_j\), only even powers of \(x_1,\dots,x_m\) remain. As the set of weights is invariant under the Weyl group of \(\SO(2m)\), and thus permutations of \(x_1,\dots,x_m\), it follows that \(\operatorname{ch}(E_\C)\) can be expressed in terms of symmetric polynomials \(\sigma_i(x_1^2,\dots,x_m^2)\), which are in turn identified with \(p_i(TM).\)

The Atiyah--Singer Index Theorem states that
if $M$ is a closed spin manifold and \(E\to M\) is  a complex vector bundle, the index of the Dirac operator on spinors of $M$ twisted with \(E\) is given by 
\begin{equation}\label{eq:AME}
	\hat{A}(M,E)=\big<\hat{A}(TM)\cdot \operatorname{ch}(E),[M]\big>,
\end{equation} 
see e.g.~\cite[Thm.~13.10]{lm-book}. Note that \(\hat{A}(M,E\oplus F)=\hat{A}(M,E)+\hat{A}(M,F)\). 

\subsection{Modular forms}\label{sec: modular forms}
Denote by \(\mathfrak{H}\subset \C\) the open upper half-plane, and let \(\Gamma\subset \SL(2,\Z)\) be a subgroup of finite index.  A \emph{modular form of weight} \(k\in\Z\) on \(\Gamma\) is a holomorphic function \(f\colon \mathfrak{H}\to\C\) obeying the equivariance property 
\[\phantom{\text{for} \begin{pmatrix} a & b\\ c&d\end{pmatrix} \in\Gamma,}
f\left(\frac{a\tau+b}{c\tau+d}\right)=(c\tau+d)^kf(\tau), \quad \text{for all } \begin{pmatrix} a & b\\ c&d\end{pmatrix} \in \Gamma.\]
Furthermore, for all \(\begin{psmallmatrix} a & b\\ c&d\end{psmallmatrix} \in \SL(2,\Z)\), the function 
\begin{equation}\label{eq: expansion}\tau\mapsto (c\tau+d)^{-k}f\left(\frac{a\tau+b}{c\tau+d}\right)\end{equation}
must have a Fourier expansion of the form \(\sum_{\ell=0}^\infty a_\ell e^{2\pi i \tau \ell/N}\) for some positive integer \(N\). For \(\begin{psmallmatrix} a & b\\ c&d\end{psmallmatrix}=\begin{psmallmatrix} 0 & -1\\ 1&0\end{psmallmatrix}\), that Fourier expansion is called the \emph{expansion at 0}, and we define \(\textrm{ord}_0(f)=\ell_0/N\), where \(\ell_0\) is the smallest integer such that \(a_{\ell_0}\neq 0\) in the expansion at 0. Similarly, \(\textrm{ord}_\infty(f)\) is defined using the Fourier expansion of \(f(\tau)\), i.e., choosing \(\begin{psmallmatrix} a & b\\ c&d\end{psmallmatrix}=\begin{psmallmatrix} 1 & 0\\ 0&1\end{psmallmatrix}\). For \(\tau\in\mathfrak{H}\), the order of vanishing of \(f\) at \(\tau\) is defined in the usual manner and denoted by \(\textrm{ord}_\tau(f)\). 
The vector space of modular forms of weight \(k\) on \(\Gamma\) is denoted \(M_k(\Gamma)\); taken together, they form an algebra \(M_*(\Gamma)=\bigoplus_{k}M_k(\Gamma).\)  

For a vector bundle \(E\), consider the formal power series in the variable $t$,
\begin{equation}\label{eq:symwedget}
\begin{aligned}
\Sym_tE&=1+E\,t+\Sym^2E\, t^2+ \dots \\	
\wedge_tE&=1+E\,t+\wedge^2E\, t^2+ \dots
\end{aligned}
\end{equation}
with coefficients given by the symmetric and exterior powers of \(E\). 

The \emph{Witten genus} of a closed oriented manifold \(M^{4k}\) is the formal power series 
\begin{equation}\label{eq:witten-genus}
\varphi_W(M)=\textstyle \hat{A}\left(\!M,\;\bigotimes\limits_{\ell=1}^\infty\Sym_{q^\ell}TM_\C\right)\prod\limits_{\ell=1}^\infty(1-q^\ell)^{4k}.
\end{equation}
The notation indicates that we apply \(\hat{A}(M,\cdot)\) to the coefficients of the given formal power series of bundles, cf.~\eqref{eq:AME}. If \(M^{4k}\) is a closed spin manifold with \(p_1(TM)=0\), then for \(q=e^{2\pi i \tau}\), the series \eqref{eq:witten-genus} is the Fourier expansion of a modular form \(\varphi_W(M)(\tau)\) of weight \(2k\) on \(\SL(2,\Z)\), see \cite[Sec.~6.3]{HBJ}. 

The \emph{elliptic genus} of a closed oriented manifold \(M^{4k}\) is the formal power series
\[\varphi(M)=\textstyle\left(\!2\!\prod\limits_{\ell=1}^\infty\frac{(1-q^\ell)^2}{(1+q^\ell)^2}\!\right)^{2k}\!\!\left<\!L(TM)\cdot\operatorname{ch}\!\left(\!\Psi_2\!\left(\bigotimes\limits_{\ell=1}^\infty\Sym_{q^\ell}TM_\C \otimes \wedge_{q^\ell}TM_\C\!\right)\!\!\right)\!,[M]\!\right>,\]
where $\Psi_2$ is an Adams operation, see \cite[p.~75]{HBJ}. Again setting \(q=e^{2\pi i \tau}\), the series \(\varphi(M)(\tau)\) is a modular form of weight \(2k\) on 
\begin{equation}\label{eq: gamma0}
\Gamma_0(2)=\left\{\begin{pmatrix}a&b\\c&d\end{pmatrix} \in \SL(2,{\Z}) : c\equiv 0\mod 2\right\}.
\end{equation}
Here, however, we make no restriction on \(p_1(TM)\). Indeed, the elliptic genus defines a ring homomorphism \(\varphi\colon \Omega_*^{\SO}\otimes\mathds{Q}\to M_*(\Gamma_0(2)).\)

Furthermore, the modified elliptic genus
\begin{equation}\label{eq: phi_phi_tilde}
\widetilde\varphi(M)(\tau)=\tau^{-2k}\varphi(M)\left(-\tfrac{1}{\tau}\right)
\end{equation}
has a Fourier expansion with \(N=2\), which is the expansion at 0 of \(\varphi\). Indeed, setting \(q=e^{2\pi i \tau}\) as before, one obtains
\begin{equation}\label{tildephi}
\widetilde\varphi(M)(2\tau)=\textstyle\left(\prod\limits_{\ell=1}^\infty\frac{(1-q^{2\ell})^4}{(1-q^\ell)^2}\right)^{2k}\!\hat{A}\left(\!M,\,\bigotimes\limits_{\ell=1}^\infty\wedge_{-q^{2\ell-1}}TM_\C\otimes\Sym_{q^{2\ell}}TM_\C\right)
\end{equation}
and the function \(\tau\mapsto \widetilde\varphi(M)(2\tau)\) is again an element of \(M_{2k}(\Gamma_0(2))\). Note that 
\begin{equation}\label{eq:phi-interp}
\lim_{t\to\infty}\varphi(M)(it)=4^kL(M), \quad\text{and}\quad \lim_{t\to\infty}(it)^{-2k}\varphi(M)\left(-\tfrac{1}{it}\right)=\hat{A}(M),	
\end{equation}
so \(\varphi(M)\) interpolates between $L$ and $\hat A$; for details, see \cite[Sec.~6.1, 6.2]{HBJ}.

\subsection{Surgery stability}
Given a manifold \(M^n\) with an embedding of \(\Ss^{n-d}\times D^d\), that is, an embedded sphere with a trivialization of its normal bundle, we can remove the embedded submanifold and glue the result to \(D^{n-d+1}\times \Ss^{d-1},\) forming 
\[N=(M\setminus \Ss^{n-d}\times D^d)\cup_{\Ss^{n-d}\times \Ss^{d-1}}D^{n-d+1}\times \Ss^{d-1},\]
which is cobordant to the original manifold $M$.
This process is referred to as a \emph{surgery of dimension} \(n-d\), or \emph{codimension }\(d\).
Surgery of dimension \(n-d\) decreases the Betti number \(b_{n-d}\) if the embedding \(\Ss^{n-d}\subset M\) is nontrivial in rational homology, and increases \(b_{n-d+1}\) if \(\Ss^{n-d}\subset M\) is trivial in rational homology.

If \(M\) has $\scal>0$ and \(d\geq 3\), then \(N\) also admits a Riemannian metric with $\scal>0$, by the celebrated works~\cite{schoen-yau,gromov-lawson}. In general, a curvature condition \(C\) is called \emph{stable under surgeries of codimension \(d\)} if \(N\) admits a metric satisfying \(C\) whenever it can be constructed using surgery of codimension \(d\) from a manifold \(M\) satisfying \(C\). Let \(\Sym^2_b(\wedge^2\R^n)\) be the space of algebraic curvature operators, that is, symmetric endomorphisms \(R\colon\wedge^2\R^n\to\wedge^2\R^n\) that satisfy the first Bianchi identity.  
The following is a far-reaching generalization of surgery stability for $\scal>0$.

\begin{theorem}[Hoelzel~\cite{hoelzel}]\label{hoelzel}
Let \(C\) be an open convex \(\mathsf O(n)\)-invariant cone in \(\Sym^2_b(\wedge^2\R^n)\), and $R_d$ be the curvature operator of \(\R^{n-d+1}\times \Ss^{d-1},\) \(3\leq d\leq n\), with its standard product metric. If \(R_d\in C\), then the condition \(R\in C\) is stable under surgeries of codimension \(d\).
\end{theorem}        

\subsection{Representation theory}\label{sec: rep-theory}
Let $\G\subset\SO(n)$ be a connected compact real Lie subgroup with Lie algebra $\mathfrak g\subset\mathfrak{so}(n)$.
An irreducible complex $\G$-representation $\pi\colon \G\to\operatorname{Aut}(E)$ is called of {real},  {complex}, or  {quaternionic} type according to whether it arises from a real $\G$-representation by extension of scalars (\emph{real type}), from a quaternionic $\G$-representation by restriction of scalars (\emph{quaternionic type}), or none of the above (\emph{complex type}). Real and quaternionic types are respectively equivalent to the existence of a conjugate-linear endomorphism which squares to $+\operatorname{Id}$ and $-\operatorname{Id}$.
In particular, an irreducible real representation $\pi\colon\G\to\operatorname{Aut}(E)$ is such that the complexified representation $\pi\colon\G\to\operatorname{Aut}(E_\C)$ is irreducible if and only if the latter is of real type.
If, instead, $\pi\colon\G\to\operatorname{Aut}(E_\C)$ is reducible, then $E_\C\cong V\oplus V^*$ for some irreducible $\G$-representation $V$, which satisfies $V^*\cong V$ if and only if $\pi$ is of quaternionic type. 

Consider the complexification $\G_\C\subset \SO(n,\C)$, whose Lie algebra is $\mathfrak g_\C\subset\sonc$. Given a (real or complex) representation $\pi\colon\G\to\operatorname{Aut}(E)$, we extend its linearization $\dd\pi\colon\mathfrak g\to\operatorname{End}(E)$ to a $\mathfrak g_\C$-representation also denoted $\dd\pi$.
Throughout this paper, $\sonc\cong \wedge^2\C^n$ and its Lie subalgebras are endowed with the inner product 
$\langle X,Y\rangle =\frac12\operatorname{Re} \operatorname{tr} XY^*$.
Fix a Cartan subalgebra $\mathfrak h\subset\mathfrak g_\C$, and identify the subspace $\mathfrak h_0^*\subset\mathfrak h^*$ spanned by the roots of $\mathfrak g_\C$ with a subspace $\mathfrak h_0\subset\mathfrak h$. Given a choice of simple roots, let $\omega_\ell$, $1\leq \ell\leq \operatorname{rk}(\mathfrak g)$, be the \emph{fundamental weights} of $\mathfrak g_\C$, i.e., the basis of $\mathfrak h_0$ dual to the basis of coroots. Let $\rho$ be the \emph{half-sum of positive roots} in $\mathfrak g_\C$, also called the \emph{Weyl vector}, and recall that $\rho=\sum_{\ell=1}^{\operatorname{rk}(\mathfrak g)} \omega_\ell$. Note that $w_0\,\rho=-\rho$, where $w_0$ the unique element of the Weyl group that sends the positive Weyl chamber to the negative one.
If $\lambda\in\mathfrak h^*$ is the highest weight of a $\mathsf G_\C$-representation $V$, then the highest weight of the dual representation $V^*$ is $-w_0\,\lambda$. Thus, for simplicity, we shall refer to the \emph{highest weight} $\lambda\in\mathfrak h^*$ of $\pi\colon\G\to\operatorname{Aut}(E)$ as being the highest weight of $E$ if it is complex, of $E_\C$ if $E$ is real and $E_\C$ is of real type, and of the complex $\G$-representation $V$ such that $E_\C\cong V\oplus V^*$ if $E$ is real and $E_\C$ is of complex or quaternionic type. There will be no ambiguity in the latter case, since throughout the paper we only use the quantities $\|\lambda\|^2=\|w_0\,\lambda\|^2$ and $\langle\lambda,\lambda+2\rho\rangle=\langle-w_0\,\lambda,-w_0\,\lambda+2\rho\rangle$ associated to $\lambda$.  

By the Highest Weight Theorem, there is a bijection between finite-dimensional irreducible representations of $\mathfrak g_\C$ and the set $P_{++}(\mathfrak g_\C)$ of dominant $\mathfrak g_\C$-integral weights. Given $\lambda\in P_{++}(\mathfrak g_\C)$, we denote by $\dd\pi_\lambda\colon\mathfrak g_\C\to\operatorname{End}(E)$ the unique (up to isomorphisms) irreducible $\mathfrak g_\C$-representation with highest weight $\lambda$.
Dominant $\G_\C$-integral weights form a sublattice $P_{++}(\G_\C)\subset P_{++}(\mathfrak g_\C)$, and given $\lambda\in P_{++}(\G_\C)$, we denote by $\pi_\lambda\colon\G_\C\to\operatorname{Aut}(E)$ the unique (up to isomorphisms) irreducible $\G_\C$-representation whose linearization is $\dd\pi_\lambda\colon\mathfrak g_\C\to\operatorname{End}(E)$. 

Given an orthonormal basis $\{\alpha_i\}$ of $\mathfrak g_\C$, define the \emph{Casimir element} $\Cas=\sum_i \alpha_i^2$ in the universal enveloping algebra $\mathcal U(\mathfrak g_\C)$. Since $\Cas$ lies in the center of $\mathcal U(\mathfrak g_\C)$, by Schur's Lemma, given an irreducible $\mathfrak g_\C$-representation $\dd\pi_\lambda$ on the vector space $E$, the operator $\dd\pi_\lambda(-\Cas)=-\sum_i \dd\pi_\lambda(\alpha_i)^2$ acts on $E$ as multiplication by a scalar, which is equal to $\langle \lambda,\lambda+2\rho\rangle$ by Freudenthal's formula, see e.g.~\cite[Lem.~5.6.4]{wallach-book}.

\begin{example}[Type $D_m$]\label{ex:reptheorySO2m}
Consider the Lie groups $\SO(2m)$ and $\Spin(2m)$, of rank $\operatorname{rk}(\G)=m\geq3$, whose complexified Lie algebras are isomorphic to $\mathfrak g_\C=\mathfrak{so}(2m,\C)$. Fix the Cartan subalgebra $\mathfrak h=\{ H(\theta_1,\dots,\theta_m)\in\mathfrak g_\C : \theta_j\in\C \}$, where 
\begin{equation*}
H(\theta_1,\dots,\theta_m):=\diag\left(
\begin{bsmallmatrix}
0 & \theta_1\\
-\theta_1 & 0
\end{bsmallmatrix},\begin{bsmallmatrix}
0 & \theta_2\\
-\theta_2 & 0
\end{bsmallmatrix},\dots,\begin{bsmallmatrix}
0 & \theta_m\\
-\theta_m & 0
\end{bsmallmatrix} \right),
\end{equation*}
and let $\varepsilon_i\in\mathfrak h^*$ be the functionals defined by $\varepsilon_i(H(\theta_1,\dots,\theta_m))=\theta_i$.
According to the fixed inner product $\langle\cdot,\cdot\rangle$ in $\mathfrak{so}(2m,\C)$, we have
$\langle\varepsilon_i,\varepsilon_j\rangle=\delta_{ij}$. Note that $\mathfrak h_0=\mathfrak h$. 
We choose $\varepsilon_i\pm \varepsilon_j$, $i<j$, as positive roots,
and $\varepsilon_1-\varepsilon_2,\dots, \varepsilon_{m-1}-\varepsilon_m, \varepsilon_{m-1}+\varepsilon_m$ as simple roots.
The fundamental weights $\omega_1,\dots,\omega_m$ are given by:
\begin{equation*}
\omega_\ell=\begin{cases}
\varepsilon_1+\dots+\varepsilon_\ell, & \text{ if } 1\leq \ell\leq m-2,\\
\tfrac{1}{2}\big( \varepsilon_1+\dots+\varepsilon_{m-1}-\varepsilon_{m} \big), & \text{ if } \ell = m-1,\\
\tfrac{1}{2}\big( \varepsilon_1+\dots+\varepsilon_{m-1}+\varepsilon_{m} \big), & \text{ if } \ell = m,
\end{cases} 
\end{equation*}
and hence the half-sum of positive roots (also called the Weyl vector) is given by
\begin{equation}\label{eq:rho}
\rho=\textstyle\sum\limits_{i=1}^m (m-i)\,\varepsilon_i.
\end{equation}

The set of dominant $\mathfrak g_\C$-integral weights is
\begin{equation*}
P_{++}(\mathfrak{g}_\C)=\left\{\lambda= \textstyle\sum\limits_{j=1}^m a_j\,\varepsilon_j  : 
\begin{array}{l}
a_j\in\Z, \, \forall j, \; \text{ or } \; a_j+\tfrac12\in\Z, \, \forall j,\\
a_1\geq a_2\geq\dots\geq a_{m-1}\geq|a_m|\geq0
\end{array}
 \right\},
\end{equation*}
and dominant $\SO(2m,\C)$-integral weights $P_{++}(\SO(2m,\C))\subset P_{++}(\mathfrak{g}_\C)$ form an index $2$ sublattice consisting of those elements with $a_j\in\Z$ for all $1\leq j\leq m$. The $\mathfrak g_\C$-representation $\dd\pi_\lambda$ is of complex type if $m$ is odd and $a_{m-1}\neq a_m$, of quaternionic type if $m\equiv 2\mod 4$ and $a_{m-1}+a_m$ is odd, and of real type otherwise. 

Let us recall certain representations in terms of irreducible representations $\pi_\lambda$, with $\lambda\in P_{++}(\G_\C)$. First, the defining representation of 
$\SO(2m)$ on $\R^{2m}$ complexifies to the irreducible $\SO(2m,\C)$-representation
$\pi_{\omega_1}\cong\pi_{\varepsilon_1}$ on $\C^{2m}$.
Exterior powers $\wedge^p\C^{2m}$ are irreducible $\SO(2m,\C)$-rep\-resentations of real type for $1\leq p\leq m-1$, and have highest weight $\omega_p=\varepsilon_1+\dots+\varepsilon_{p}$, i.e., $\wedge^p\C^{2m}\cong\pi_{\omega_p}$ if $1\leq p\leq m-2$, and $\wedge^{m-1}\C^{2m}\cong\pi_{\omega_{m-1}+\omega_m}$. However, if $p=m$, then $\wedge^m\C^{2m}\cong\wedge^m_+\C^{2m}\oplus\wedge^m_-\C^{2m}$ is not irreducible: it decomposes into the sum of $\pm1$-eigenspaces of the Hodge star operator $*$, called \emph{self-dual} and \emph{anti-self-dual} parts, which are irreducible and have highest weight $\varepsilon_1+\dots+\varepsilon_{m-1}\pm\varepsilon_{m}$, i.e., $\wedge^m_+\C^{2m}\cong\pi_{2\omega_{m}}$ and $\wedge^m_-\C^{2m}\cong\pi_{2\omega_{m-1}}$. The remaining exterior powers $m<p\leq 2m$ are identified via the isomorphisms $\wedge^{2m-p}\C^{2m}\cong\wedge^p\C^{2m}$, $1\leq p\leq 2m$, given by $*$. Altogether, 
\begin{equation}\label{eq:decomp-wedgep}
	\wedge^p \C^{2m} \cong \wedge^p\pi_{\omega_1}\cong \begin{cases}
	\pi_{\omega_p}, & 1\leq p\leq m-2,\\
	\pi_{\omega_{m-1}+\omega_m}, & p= m-1,\\
	\pi_{2\omega_{m}}\oplus \pi_{2\omega_{m-1}}, & p = m.
\end{cases}
\end{equation}
Traceless symmetric powers $\Sym^p_0\C^{2m}$ are $\SO(2m,\C)$-irreducible of real type for all $p\geq1$, and have highest weight $p\,\omega_1$. Symmetric powers $\Sym^p\C^{2m}$ decompose into the sum of traceless symmetric powers $\Sym^{p-2j}_0\C^{2m}$, $0\leq j\leq p$; or, in symbols:
\begin{equation}\label{eq:decomp-symp}
\Sym^p \C^{2m} \cong \Sym^p\pi_{\omega_1}\cong \textstyle\bigoplus\limits_{j=0}^{\floor{p/2}} \pi_{(p-2j)\omega_1}.
\end{equation}

Precomposing an $\SO(2m)$-representation $\pi \colon \SO(2m)\to \operatorname{Aut}(E)$ with the double cover $\Spin(2m)\to\SO(2m)$ gives rise to a $\Spin(2m)$-representation $\widehat\pi$. Their complexifications have the same highest weight, and the isomorphisms \eqref{eq:decomp-wedgep} and \eqref{eq:decomp-symp} remain valid when these are considered as $\Spin(2m,\C)$-representations. 

A $\Spin(2m,\C)$-representation that is not the lift of any $\SO(2m,\C)$-representation is the \emph{spinor representation} $S=S^+\oplus S^-$, which is the sum of (irreducible) positive and negative ``half'' spinor representations, each of dimension $2^{m-1}$. 
If $m$ is even, then $S^+\cong\pi_{\omega_m}$ and $S^-\cong\pi_{\omega_{m-1}}$ are self-dual and of real type if $m\equiv 0\mod 4$, quaternionic type if $m\equiv 2\mod 4$; while if $m$ is odd, then $S^+\cong\pi_{\omega_{m-1}}$ and $S^-\cong\pi_{\omega_{m}}$ are of complex type and $(S^\pm)^*\cong S^\mp$. For convenience, we often write $\pi_S:=\pi_{\omega_m}\oplus\pi_{\omega_{m-1}}$ for the spinor representation.
\end{example}

\begin{example}[Type $B_m$]\label{ex:reptheorySO2m+1}
Consider $\G=\SO(2m+1)$, which also has $\operatorname{rk}(\G)=m$. We shall use the same notation and same Cartan subalgebra $\mathfrak h\subset\mathfrak{so}(2m,\C)$ from Example~\ref{ex:reptheorySO2m} as a Cartan subalgebra of $\mathfrak g_\C=\mathfrak{so}(2m+1,\C)$ by means of an appropriate embedding $\SO(2m,\C)\subset\SO(2m+1,\C)$.
Choose $\varepsilon_i\pm \varepsilon_j$, $i<j$, and $\varepsilon_i$ as positive roots, and $\varepsilon_1-\varepsilon_2,\dots, \varepsilon_{m-1}-\varepsilon_m, \varepsilon_m$ as simple roots, and recall $\mathfrak h_0=\mathfrak h$. The fundamental weights $\omega_1,\dots,\omega_m$ are given by:
\begin{equation*}
\omega_\ell=\begin{cases}
\varepsilon_1+\dots+\varepsilon_\ell, & \text{ if } 1\leq \ell\leq m-1,\\
\tfrac{1}{2}\big( \varepsilon_1+\dots+\varepsilon_{m-1}+\varepsilon_{m} \big), & \text{ if } \ell = m.
\end{cases}
\end{equation*}
and hence the half-sum of positive roots is given by
\begin{equation}\label{eq:rhoSOodd}
\rho=\textstyle\sum\limits_{i=1}^m (m-i+\frac12)\,\varepsilon_i.
\end{equation}
The set of dominant $\SO(2m+1,\C)$-integral weights in $\mathfrak g_\C$ is
\begin{equation*}
P_{++}\big(\SO(2m+1,\C)\big)=\left\{\lambda= \textstyle\sum\limits_{j=1}^m a_j\,\varepsilon_j  : 
a_j\in\Z, \, \forall j, \; a_1\geq\dots\geq a_m\geq0 \right\}.
\end{equation*}
Similarly to Example~\ref{ex:reptheorySO2m},  $\C^{2m+1}\cong \pi_{\omega_1}\cong\pi_{\varepsilon_1}$ is the complexification of the defining representation. Exterior powers are given by $\wedge^p\C^{2m+1}\cong \pi_{\omega_p}$ for $1\leq p\leq m-1$, and $\wedge^m\C^{2m+1}\cong\pi_{2\omega_m}$. The operator $*$ induces isomorphisms $\wedge^{2m+1-p}\C^{2m+1}\cong \wedge^p\C^{2m+1}$. Symmetric powers are given by $\Sym^p_0\C^{2m+1}\cong \pi_{p\omega_1}$ for all $p\geq1$, and $\Sym^p\C^{2m+1}\cong\bigoplus_{j=0}^{\floor{p/2}} \Sym^{p-2j}_0\C^{2m+1}$. All representations above are of real type.
\end{example}

\begin{example}
Let us consider $\U(m)\subset\SO(2m)$, $m\geq2$, which is \emph{not} semisimple, but is reductive, and has $\operatorname{rk}(\U(m))=m-1$.
The complexification of its Lie algebra is $\mathfrak g_\C=\mathfrak{gl}(m,\C)\subset\mathfrak{so}(2m,\C)$, and it splits as $\mathfrak g_\C=\mathfrak g'_\C\oplus \mathfrak z(\mathfrak g_\C)$, where 
$\mathfrak g'_\C=\mathfrak{sl}(m,\C)$ is the semisimple part (of type $A_{m-1}$), and $\mathfrak z(\mathfrak g_\C)=\C \, \mathrm{Id}$ are multiples of the identity.
Fix the Cartan subalgebra $\mathfrak h=\mathfrak h_0\oplus \mathfrak z(\mathfrak g_\C)$, where $\mathfrak h_0=\{ \diag\left(
\theta_1,\dots, \theta_m \right) \in\mathfrak g'_\C : \theta_j\in\C \}$ is the Cartan subalgebra of the semisimple part, which is spanned by the roots of $\mathfrak g_\C$. Using $\varepsilon_j\big(\diag (
\theta_1,
\dots
\theta_m )\big)=\theta_j$, the fundamental weights are given by
\begin{equation*}
\omega_\ell= \varepsilon_1+\dots+\varepsilon_\ell - \tfrac{\ell}{m}(\varepsilon_1+\dots+\varepsilon_m), \quad  1\leq \ell\leq m-1,\\
\end{equation*}
and the half-sum of positive roots  is 
\begin{equation}\label{eq:rhoUm}
    \rho= \tfrac12\textstyle\sum\limits_{j=1}^m (m-2j+1) \varepsilon_j.
\end{equation}
The set of dominant integral weights is given by
\begin{equation*}
P_{++}\big(\G_\C\big)=P_{++}\big(\mathfrak g_\C\big)=\left\{\lambda=\textstyle\sum\limits_{j=1}^m a_j\varepsilon_j  \,:\,  a_j \in\Z\, \text{ and } a_1\geq a_2\geq\dots\geq a_m
 \right\}.
\end{equation*}
\end{example}

\section{Revisiting the Bochner technique with representation theory}
\label{sec:bochner}

In this section, we develop a representation-theoretic framework that 
allows us to prove an abstract Bochner-type result (Theorem~\ref{thm:bochner}) simultaneously generalizing some recent results of Petersen--Wink~\cite{petersen-wink,petersen-wink-cplx,petersen-wink-3}, see Theorems~\ref{thm:pw}, \ref{thm:pwtachibana}, and \ref{thm:pw-cplx}.
We being by recalling the construction of Weitzenb\"ock formulae.

\subsection{Weitzenb\"ock formulae}\label{sec:weitzenbock}
Let $(M^n,\g)$ be an orientable Riemannian manifold, with curvature operator $R\colon \wedge^2TM\to\wedge^2TM$. 
We denote by $\mathrm{Hol}(M^n,\g)$ the holonomy group of $(M^n,\g)$, or its lift to $\Spin(n)$ if $M$ is spin.
Let $\G$ be a connected compact Lie subgroup of $\SO(n)$, or $\Spin(n)$, that contains $\mathrm{Hol}(M^n,\g)$.
Given an orthogonal or unitary representation $\pi\colon\G\to \operatorname{Aut}(E)$, let $E_\pi:=\Fr(M)\times_\pi E$ be the associated bundle to the principal $\G$-bundle $\G\to\Fr(M)\to M$ obtained by reducing the structure group of the bundle of $\SO(n)$-frames, or $\Spin(n)$-frames if $M$ is spin.
For instance, if $\G=\SO(n)$ and $\pi$ is the defining representation on $E=\R^n$, then $E_\pi=TM$. Similarly, the representations $\wedge^p \pi$ and $\Sym^p \pi$ give rise to the bundles $\wedge^p TM$ and $\Sym^p TM$, respectively.

The curvature term in the Weitzenb\"ock formula $\Delta_\pi=\nabla^*\nabla+t\,K(R,\pi)$ for sections of $E_\pi\to M$, where $t\in\R$ is an appropriate constant, is given by
\begin{equation}\label{eq:krrho}
K(R,\pi)=-\textstyle\sum\limits_{a,b} R_{ab}\, \dd\pi(X_a)\circ\dd\pi(X_b)=-\sum\limits_a\dd\pi(R(X_a))\circ\dd\pi(X_a),
\end{equation}
where $R=\sum_{a,b}R_{ab}\,X_a\otimes X_b$ is the curvature operator of $(M^n,\g)$, and $\{X_a\}$ is an orthonormal basis of $\mathfrak g\subset\mathfrak{so}(n)\cong\wedge^2 T_pM$, see e.g.,~\cite[\S 1.I]{besse} or \cite{hitchin-bochner}.
Note that the image of $R$ is contained in $\mathfrak g$ because $\mathrm{Hol}(M^n,\g)\subset\G$, so we may consider $R\colon\mathfrak g\to\mathfrak g$.
To simplify notation, we also denote \eqref{eq:krrho} by $K(R,E_\pi)$.
The self-adjoint extension of $R$ to $\mathfrak g_\C$, and of $K(R,\pi)$ to $(E_\pi)_\C$, are denoted by the same symbols.
Note that the construction \eqref{eq:krrho} can be performed  with $\dd\pi$ as a $\mathfrak g$-representation or a $\mathfrak g_\C$-representation, in which case $\{X_a\}$ is taken to be an orthonormal basis of $\mathfrak g_\C$.

\begin{proposition}\label{prop:basicsKRpi}
The endomorphisms \eqref{eq:krrho} satisfy the following basic properties:
    \begin{enumerate}[\rm (i)]
        \item The linear map $R\mapsto K(R,\pi)$ is $\G$-equivariant, and $K(R,\pi)$ is self-adjoint;
        \item If $R\succeq0$, then $K(R,\pi)\succeq0$ for any orthogonal or unitary $\G$-representation $\pi$;  
        \item If $\pi$ is reducible, say $\pi\cong\pi'\oplus\pi''$, then $K(R,\pi)=\diag\big(K(R,\pi'),K(R,\pi'')\big)$ is block diagonal according to the  decomposition $E_\pi=E_{\pi'}\oplus E_{\pi''}$;
        \item If $\pi^*\colon \G\to\operatorname{Aut}(E^*)$ is the dual of $\pi\colon\G\to\operatorname{Aut}(E)$, then $K(R,\pi^*)=K(R,\pi)^*$;
        \item If $\lambda\in P_{++}(\G_\C)$, then $K(\operatorname{Id},\pi_\lambda)=\dd\pi_\lambda(-\Cas)=\langle \lambda,\lambda+2\rho\rangle\operatorname{Id}$.
    \end{enumerate}
\end{proposition}

The proofs are elementary and left to the reader. Let us discuss a few examples:

\begin{example}\label{ex:vanillaDirac}
Let $\pi_S$ be the spinor representation, see Example~\ref{ex:reptheorySO2m}. A standard computation using the symmetries of Clifford multiplication and the (first) Bianchi identity yields $K(R,\pi_S)=\frac{\scal}{8}\operatorname{Id}$, see \cite[Thm.~II.8.8]{lm-book}. The square of the Dirac operator $D$ on the complex spinor bundle $S=E_{\pi_S}$ of a spin manifold $(M^n,\g)$ satisfies $D^2=\nabla^*\nabla+t\,K(R,\pi_S)$ with $t=2$, cf.~\eqref{eq:Lichn}.
\end{example}

\begin{example}\label{ex:TM}
If $\pi$ is either the defining representation of $\SO(n)$ on $\R^n$, or its dual, then $K(R,\pi)=\Ric$, see \cite[Ex.~2.2]{bm-iumj} and Proposition~\ref{prop:basicsKRpi} (iv). The Hodge Laplacian on the  bundle $TM^*$ of $1$-forms is $\Delta_\pi=\nabla^*\nabla+t\,K(R,\pi)$ with $t=2$.
\end{example}

\begin{example}\label{ex:trivial}
If $\pi=\bigoplus_i \pi_{i}$ is a decomposition 
into the direct sum of irreducibles, then 
the associated bundle $E_\pi$ decomposes into the corresponding direct sum of subbundles $E_\pi=\bigoplus_i E_{\pi_{i}}$.
We denote by $(E_\pi)_0\subset E_\pi$ the subbundle corresponding to the \emph{trivial} isotypic component, and write $E_\pi=(E_\pi)_0\oplus (E_\pi)_0^\perp$. Note that $(E_\pi)_0\to M$ is a trivial bundle, i.e., if the trivial isotypic component of $\pi$ consists of $q$ copies of the trivial representation, then $(E_\pi)_0=M\times \R^q$. Clearly, $K(R,(E_\pi)_0)=0$ by \eqref{eq:krrho}, and thus $K(R,E_\pi)=\diag(0, K(R,(E_\pi)_0^\perp))$ by Proposition~\ref{prop:basicsKRpi} (iii).
\end{example}

Recall the decomposition $R=R_{\mathcal U}+R_{\mathcal L}+R_{\mathcal W}+R_{\wedge^4}$ of $R\in\Sym^2(\wedge^2\R^n)$ into $\mathsf O(n)$-irreducible components, see \cite[\S 1.G]{besse}. In particular,
\begin{equation}\label{eq:RURL}
\textstyle R_\mathcal U=\frac{\scal}{2n(n-1)} \,\g\owedge\g, \qquad R_\mathcal L=\frac{1}{n-2}\, \g\owedge\left(\Ric-\frac{\scal}{n}\right)\!,
\end{equation}
where $\owedge$ is the Kulkarni--Nomizu product. 
The Weyl part $R_\mathcal W$ does not have $\g$ factors, nor does $R_{\wedge^4}$, which vanishes if and only if $R$ satisfies the Bianchi identity.

\begin{example}\label{ex:Labbi-formulas}
The exterior and symmetric $p^\text{th}$ powers of the defining representation, respectively of its dual, give rise to bundles $E_\pi$ which are isomorphic to $\wedge^p TM$ and $\Sym^p TM$, respectively $\wedge^p TM^*$ and $\Sym^p TM^*$.
By Proposition~\ref{prop:basicsKRpi}~(iv), we only consider the former. Note that $K(R,\wedge^p TM)$ and $K(R,\Sym^p TM)$ are block diagonal according to the decompositions into irreducibles in Examples~\ref{ex:reptheorySO2m} and \ref{ex:reptheorySO2m+1}, by Proposition~\ref{prop:basicsKRpi} (iii).
These blocks can be computed in terms of the decomposition of $R\in\Sym^2(\wedge^2\R^n)$ into $\mathsf O(n)$-irreducible components (see \cite[Thm B]{bm-iumj}):
\begin{equation*}
\begin{aligned}
K\big(R,\wedge^p\pi_{\omega_1}\big)&=\left(\tfrac{2(n-p)}{p-1}\,R_{\mathcal U}+\tfrac{n-2p}{p-1}\,R_{\mathcal L}-2\,R_{\mathcal W}+4\,R_{\wedge^4} \right)\owedge \tfrac{\g^{\owedge(p-2)}}{(p-2)!},\\
K(R, \Sym^p_0 \pi_{\omega_1})&=\Big(\tfrac{n+p-2}{n(p-1)}  K(R_{\mathcal U},\pi_{2\omega_1})+\tfrac{n+2p-4}{n(p-1)}  K(R_{\mathcal L},\pi_{2\omega_1}) \\ &\qquad\qquad\qquad\qquad\qquad\qquad\qquad\qquad +  K(R_{\mathcal W},\pi_{2\omega_1}) \Big)\ovee\tfrac{\g^{\ovee(p-2)}}{(p-2)!},\\
\end{aligned}
\end{equation*}
for all $p\geq2$ and $2\leq p\leq n-2$, respectively, where $\ovee$ is a symmetric version of the Kulkarni--Nomizu product.
If $p=2$, from \cite[Thm B]{bm-iumj} and \cite[Eq.~(22)]{hms},
\begin{equation*}
\begin{aligned}
    K(R,\wedge^2\pi_{\omega_1}) & = 2(n-2) R_{\mathcal U} + (n-4) \,R_{\mathcal L}-2\,R_{\mathcal W}+4\,R_{\wedge^4},\\
    \big(K(R,\Sym^2\pi_{\omega_1})\varphi\big) (X,Y)&=\varphi(\Ric X,Y)+\varphi(X,\Ric Y)-2(\mathring R\varphi)(X,Y),
\end{aligned}
\end{equation*}
where $\mathring R\colon\Sym^2\R^n\to\Sym^2\R^n$ is given by (cf.~\cite[p.~74]{bk} and \cite[p.~52]{besse})
\begin{equation*}
(\mathring R\varphi)(X,Y)=\textstyle\sum\limits_{i,j} \langle R(e_i\wedge X), e_j\wedge Y\rangle \varphi(e_i , e_j).
\end{equation*}
Note that $K(R,\Sym^2\pi_{\omega_1})$ vanishes on the subspace spanned by $\g$, since $\mathring R\g=\Ric$,\newline so $K(R,\Sym^2\pi_{\omega_1})=\operatorname{diag}\big(K(R,\pi_{2\omega_1}),0\big)$ according to the splitting $\Sym^2\pi_{\omega_1}\cong \pi_{2\omega_1}\oplus \pi_0$ of symmetric $2$-tensors into traceless symmetric $2$-tensors and multiples of the identity, respectively.

The Laplacian $\Delta_\pi=\nabla^*\nabla+t\,K(R,\pi)$ is the Hodge Laplacian on $p$-forms if $\pi=\wedge^p\pi_{\omega_1}^*$ and $t=2$, and is the Lichnerowicz Laplacian on symmetric $p$-tensors if $\pi=\Sym^p\pi_{\omega_1}^*$ and $t=-2$. 
\end{example}

\subsection{\texorpdfstring{Sufficient condition for $K(R,\pi)\succeq0$}{Sufficient condition for positivity of curvature term}}
Consider the following:

\begin{definition}\label{def:sigmarR}
Given a self-adjoint operator \(R\) of a $d$-dimensional vector space, whose eigenvalues are $\nu_1\leq \nu_2\leq \dots\leq \nu_d$, we define for each real number $0<r\leq d$,
\begin{equation*}
 \Sigma(r,R)=\nu_1 + \dots + \nu_{\floor{r}} + (r-\floor{r})\nu_{\floor{r}+1}.
\end{equation*} 
We say \(R\) is \emph{$r$-nonnegative} if \(\Sigma(r,R)\geq0\), and \emph{$r$-positive} if $\Sigma(r,R)>0$. Similarly, $R$ is \emph{$r$-nonpositive}, or \emph{$r$-negative}, if the operator $-R$ is $r$-nonnegative, or $r$-positive.  
\end{definition}

For instance, $R$ is $1$-positive if and only if $R\succ0$, and $\frac43$-positive if $\nu_1+\frac13\nu_2>0$.

Note that \(\Sigma(r,R)/r\) is the (continuous) arithmetic mean of the smallest $r$ eigenvalues of $R$, hence nondecreasing in \(r\); and that \(-\Sigma(r,-R)\) is a sum involving the \(\floor{r}+1\) largest eigenvalues of \(R\). Moreover, $\Sigma(r,R)$ and $\Sigma(r,-R)$ are concave in $R$.

Motivated by the key algebraic method underlying the recent works of Petersen--Wink~\cite{petersen-wink,petersen-wink-cplx,petersen-wink-3}, we introduce a representation-theoretic invariant:

\begin{definition}\label{def:PW}
Let $\G$ be a connected compact Lie subgroup
of $\SO(n)$ or $\Spin(n)$, with Lie algebra $\mathfrak g\subset\mathfrak{so}(n)$, and let $\rho$ be the half-sum of positive roots in $\mathfrak g_\C$. 

The \emph{Petersen--Wink invariant} of a nontrivial irreducible orthogonal or unitary $\G$-representa\-tion $\pi$ with highest weight $\lambda\in P_{++}(\G_\C)$ is the positive real number
\begin{equation*}
PW_\G(\pi)=\frac{\langle \lambda,\lambda+2\rho\rangle}{\|\lambda\|^2},
\end{equation*}
and, in case $\frac{\langle \lambda,\lambda+2\rho\rangle}{\|\lambda\|^2}\geq\dim\mathfrak{g}$, we use the convention that $PW_\G(\pi)=\dim\mathfrak{g}$. 
Furthermore, given the decomposition $\pi=\bigoplus_i\pi_{i}$ of an orthogonal or unitary $\G$-representation into irreducibles, we set $PW_\G(\pi):=\min \{PW_\G(\pi_{i}):\pi_i \text{ is nontrivial}\}$.
\end{definition}

\begin{remark}
	Given any connected compact Lie group $\G$, there exist constants $c >0$ and $C >0$ such that $c\, \|\lambda\|^2\leq\langle \lambda,\lambda+2\rho\rangle \leq C\, \|\lambda\|^2$ for all $\lambda\in P_{++}(\G_\C)$, see e.g.~\cite[Lemma 5.6.6]{wallach-book}. For all Lie groups $\G$ considered in this paper, $C$ can be chosen such that $C  < \dim \mathfrak g$, see Remark~\ref{rem:maxPW-SOn} for $\G=\SO(n)$ and $\Spin(n)$. However, it is unclear to us if this holds in general, so we use the convention in Definition~\ref{def:PW}.
\end{remark}

The significance of $PW_\G(\pi)$ to the Bochner technique hinges on the following result, which implies Theorem~\ref{mainthm:pw} in the Introduction: 

\begin{proposition}\label{lowerbound}
	Let $\G$ be as above, \(\pi\colon \G\to\operatorname{Aut}(E)\) be a nontrivial irreducible orthogonal or unitary $\G$-representation with highest weight $\lambda$, and $R\colon\mathfrak g\to\mathfrak g$ be a self-adjoint operator. Then
$K(R,\pi )\succeq\|\lambda\|^2\,\Sigma(PW_\G(\pi),R)\operatorname{Id}$.
\end{proposition}

As \(\lambda=0\) only for the trivial representation, Propositions \ref{lowerbound} and \ref{prop:basicsKRpi} (iii) imply:

\begin{corollary}\label{pointwise}
Let $\G$ be as above, \(\pi\colon \G\to\operatorname{Aut}(E)\) be an orthogonal or unitary $\G$-representation, and $R\colon\mathfrak g\to\mathfrak g$ be a self-adjoint operator.
\begin{enumerate}[\rm (i)]
    \item If $R$ is $PW_\G(\pi)$-nonnegative, then $K(R,\pi)\succeq0$. 
    \item If $R$ is $PW_\G(\pi)$-positive and $\pi$ has no trivial components, then \(K(R,\pi)\succ0\).
\end{enumerate}
\end{corollary}

The following proof incorporates some arguments from \cite{petersen-wink} and \cite[Lemma 3.2 and Remark 3.3]{nienhaus-petersen-wink}.

\begin{proof}[Proof of Proposition~\ref{lowerbound}]
 Let \(\{X_i\}\) be an orthonormal basis of \(\mathfrak{g}\) which diagonalizes \(R\), so that \(R(X_i)=\nu_i\,X_i\), with \(\nu_i\in\R\) and $\nu_1\leq\nu_2\leq\dots\leq\nu_{\dim\mathfrak{g}}.$

 First, assume \(PW_\G(\pi)=\frac{\langle \lambda,\lambda+2\rho\rangle}{\|\lambda\|^2}<\dim\mathfrak{g}\). Let \(v\in E\) and let \(r=\floor{PW_\G(\pi)}\). Since \(\pi\) is orthogonal or unitary, \(\dd\pi(X_i)\) is anti-self-adjoint, and by \eqref{eq:krrho} we have
\begin{align*}
\left<K(R,\pi)v,v\right>&=\textstyle\sum_i\left<\dd\pi(R(X_i))(v),\dd\pi(X_i)v\right>\\
&=\textstyle\sum_i\nu_i\left\|\dd\pi(X_i)v\right\|^2\\
&\geq\textstyle\sum\limits_{i=1}^r\nu_i\left\|\dd\pi(X_i)v\right\|^2+\nu_{r+1}\sum\limits_{i=r+1}^{\dim\mathfrak{g}}\left\|\dd\pi(X_i)v\right\|^2\\
&=-\textstyle\sum\limits_{i=1}^r(\nu_{r+1}-\nu_i)\left\|\dd\pi(X_i)v\right\|^2+\nu_{r+1}\sum\limits_{i=1}^{\dim\mathfrak g}\left\|\dd\pi(X_i)v\right\|^2\\
&\geq-\textstyle\sum\limits_{i=1}^r(\nu_{r+1}-\nu_i)\|\lambda\|^2\|v\|^2+\nu_{r+1}\left<\lambda,\lambda+2\rho\right>\|v\|^2.
\end{align*}
The last inequality above follows from Proposition~\ref{prop:basicsKRpi} (v) and the observation that \(\|\dd\pi(X_i)v\|\leq\|\lambda\|\|v\|\). 
Indeed, up to conjugating, we may assume that \(X_i\) is in a given Cartan subalgebra $\mathfrak h\subset\mathfrak g_\C$. The eigenvalues of \(\dd\pi(X_i)\) are given by \(\mu(X_i)\), for each weight \(\mu\in\mathfrak{h}^*\) of the complexification of \(\pi\). Since \(\lambda\) has maximal length among such weights and \(\|X_i\|=1\), we have \(|\mu(X_i)|\leq\|\lambda\|\), so the observation follows.

Applying Definitions~\ref{def:sigmarR} and \ref{def:PW} to the above inequality, we have 
\begin{align*}
\left<K(R,\pi)v,v\right> &\geq\|\lambda\|^2\left(\textstyle\sum\limits_{i=1}^r\nu_i+\left(PW_\G(\pi)-r\right)\nu_{r+1}\right)\|v\|^2 \\
&= \|\lambda\|^2 \, \Sigma({PW_\G(\pi)},R)\, \|v\|^2.
\end{align*}
Second, suppose \(PW_\G(\pi)={\dim\mathfrak g}\leq\frac{\langle \lambda,\lambda+2\rho\rangle}{\|\lambda\|^2}\). If \(\nu_{\dim\mathfrak g}\geq0,\) a similar argument holds, using $\nu_{\dim \mathfrak g}$ instead of $\nu_{r+1}$ in the second sum above. If \(\nu_{\dim\mathfrak g}<0,\) and thus all \(\nu_i<0,\) the conclusion follows directly from \(\|\dd\pi(X_i)v\|\leq\|\lambda\|\|v\|.\)
\end{proof}

\begin{remark}
Proposition~\ref{lowerbound} is sharp. Indeed, setting $R=\operatorname{Id}$, it follows from Proposition~\ref{prop:basicsKRpi} (v) that $K(\operatorname{Id},\pi )=\langle \lambda,\lambda+2\rho\rangle \,\operatorname{Id}=\|\lambda\|^2\,\Sigma(PW_\G(\pi),\operatorname{Id})\operatorname{Id}$.
\end{remark}

Let us compute the Petersen--Wink invariant for some classes of representations.

\begin{proposition}\label{prop:PWforSOn}
Let $\pi$ be the irreducible orthogonal or unitary $\SO(n)$-represen\-tation with highest weight $\lambda=\sum\limits_{j=1}^m a_j\varepsilon_j$, as in Examples~\ref{ex:reptheorySO2m} and \ref{ex:reptheorySO2m+1}, $n\geq6$. Then
   \begin{equation}\label{eq:PWforSOn}
    \phantom{ \qquad m=\floor{\tfrac{n}{2}}}
        PW_{\SO(n)}(\pi)=1+\dfrac{\sum_{j=1}^m (n-2j)\,a_j}{a_1^2+\dots+a_m^2}, \qquad m=\floor{\tfrac{n}{2}}.
    \end{equation} 
\end{proposition}

\begin{proof}
It follows from \eqref{eq:rho}, \eqref{eq:rhoSOodd}, and Definition~\ref{def:PW}, 
that $PW_{\SO(n)}(\pi)$ equals
\begin{equation*}
        \begin{cases}
        1+\tfrac{(2m-2)a_1+(2m-4)a_2+\dots+4a_{m-2}+2a_{m-1}}{a_1^2+\dots+a_m^2}, & \text{if } n=2m,\\[10pt]
        1+\tfrac{(2m-1)a_1+(2m-3)a_2+\dots+3a_{m-1}+a_{m}}{a_1^2+\dots+a_m^2}, & \text{if } n=2m+1,
\end{cases}
\end{equation*}
and the above simplifies simultaneously for $n=2m$ and $n=2m+1$ to \eqref{eq:PWforSOn}.
\end{proof}

In particular, by Proposition~\ref{prop:PWforSOn} and Examples~\ref{ex:reptheorySO2m} and \ref{ex:reptheorySO2m+1} if $n\geq6$, and by direct computation using Definition~\ref{def:PW} if $2\leq n\leq6$, we have that, given any $n\geq2$,
\begin{align}
    PW_{\SO(n)}(\wedge^p \pi_{\omega_1}) &= n-p, &&\text{for all } 1\leq p\leq\floor{\tfrac{n}{2}},\label{eq:pw-wedgep}\\
    PW_{\SO(n)}(\Sym^p\pi_{\omega_1}) &= \tfrac{n+p-2}{p}, &&\text{for all } p\geq1.\label{eq:pw-symp}
\end{align}
Clearly, we have $PW_{\Spin(n)}(\widehat{\pi})=PW_{\SO(n)}(\pi)$, where $\widehat{\pi}$ is the precomposition of the $\SO(n)$-representation $\pi$ with the double cover $\Spin(n)\to\SO(n)$. Moreover, the Petersen--Wink invariant $PW_{\Spin(n)}(\pi)$ of an irreducible orthogonal or unitary $\Spin(n)$-represen\-tation $\pi$ with highest weight $\lambda=\sum_{j=1}^m a_j\varepsilon_j$ is also given by \eqref{eq:PWforSOn}, as the proof of Proposition~\ref{prop:PWforSOn} only requires $\lambda\in P_{++}(\mathfrak{so}(n,\C))$. Thus, we shall unambiguously write $PW(\pi)$ for both $\Spin(n)$- and $\SO(n)$-representations whose highest weight is $\lambda\in P_{++}(\mathfrak{so}(n,\C))$ as above, and compute it using~\eqref{eq:PWforSOn}.

\subsection{Vanishing theorems}
Throughout this subsection, $(M^n,\g)$ denotes a closed Riemannian $n$-manifold, $n\geq2$, with curvature operator $R\colon \wedge^2 TM\to\wedge^2 TM$. Consider the same setup as Section~\ref{sec:weitzenbock}, where $\G$ is a connected compact Lie group that contains $\operatorname{Hol}(M^n,\g)$, and recall that if $\pi\colon\G\to\operatorname{Aut}(E)$ is an orthogonal or unitary $\G$-representation, then $(E_\pi)_0\subset E_\pi$ is the subbundle corresponding to the trivial isotypic component of $\pi$, see Example~\ref{ex:trivial}. If $\pi$ has no trivial component, $(E_\pi)_0=0$.
To simplify notation, we write $PW_\G(E_\pi)=PW_\G(\pi)$, in the same way as $K(R,E_\pi)=K(R,\pi)$.
In light of Corollary~\ref{pointwise}, the Bochner technique yields:

\begin{theorem}\label{thm:bochner}
    Let $\Delta_\pi=\nabla^*\nabla + t \, K(R,\pi)$ be the Laplacian on $E_\pi\to M$, where $R\colon\mathfrak g\to\mathfrak g$ is the curvature operator of $(M^n,\g)$ restricted to $\mathfrak g\subset\mathfrak{so}(n)$. Given a harmonic section $\phi$ of $E_\pi$, i.e., $\Delta_\pi\phi=0$, the following hold:
    \begin{enumerate}[\rm (i)]
        \item If $t\geq0$ and $R$ is $PW_\G(\pi)$-nonnegative, then $\nabla \phi\equiv0$.
        \item If $t > 0$ and $R$ is $PW_\G(\pi)$-positive, then $\phi\in (E_\pi)_{0}$.
         \item If $t\leq0$ and $R$ is $PW_\G(\pi)$-nonpositive, then $\nabla \phi\equiv0$.
        \item If $t < 0$ and $R$ is $PW_\G(\pi)$-negative, then $\phi\in (E_\pi)_{0}$.
    \end{enumerate}
\end{theorem}

\begin{proof}
All assertions follow from Proposition~\ref{prop:basicsKRpi} and Corollary~\ref{pointwise}, together with
	\begin{equation}\label{eq:bochner}
		0=\int_M \langle\Delta_\pi \phi,\phi\rangle = \int_M \|\nabla\phi\|^2 + t\,\langle K(R,\pi)\phi,\phi\rangle.
	\end{equation}
Namely, let us consider only $t\geq0$. If $R$ is $PW_\G(\pi)$-nonnegative, then $K(R,\pi)\succeq0$ by Corollary~\ref{pointwise}, so $\nabla\phi\equiv0$. If $R$ is $PW_\G(\pi)$-positive, then, as in Example~\ref{ex:trivial}, we have $K(R,\pi)=\diag(0, K(R,(E_\pi)_0^\perp))$. Thus, $K(R,(E_\pi)_0^\perp)\succ0$ by Corollary~\ref{pointwise} because $PW_\G(E_\pi)=PW_\G((E_{\pi})_0^\perp)$, and hence $\phi\in (E_\pi)_0$ provided that $t>0$.
\end{proof}

\begin{remark}
    In general, $\nabla \phi\equiv0$ does not imply $\phi\in (E_\pi)_0$. Instead, $\nabla \phi\equiv0$ implies $\phi\in (E_{\pi'})_0$ where $\pi'$ is the restriction of $\pi$ to $\mathrm{Hol}(M^n,\g)$.
\end{remark}

The classical vanishing theorems due to Bochner regarding Ricci curvature follow from the above ideas and Example~\ref{ex:TM}. Namely, if $(M^n,\g)$ is a closed oriented Riemannian $n$-manifold, let $\pi$ be the dual of the defining representation of $\SO(n)$, so $E_\pi$ is the bundle of $1$-forms on $(M^n,\g)$.
From Example~\ref{ex:TM}, the Hodge Laplacian is $\Delta_\pi=\nabla^*\nabla+2\,K(R,\pi)=(\dd+\dd^*)^2=\dd\dd^*+\dd^*\dd$, and $K(R,\pi)=\Ric$. 
First, if $\Ric\succeq0$, then any harmonic $1$-form is parallel by \eqref{eq:bochner} and hence $b_1(M)\leq n$; and if $\Ric\succ0$, then harmonic $1$-forms vanish so $b_1(M)=0$.

Second, decompose $\operatorname{End}(TM)\cong TM^*\otimes TM^*\cong \wedge^2 TM^* \oplus C^\infty(M) \oplus \Sym^2_0 TM^*$ according to the decomposition into $\mathsf{SO}(n)$-irreducibles of $\pi\otimes\pi$. Given a $1$-form $\phi$, let $T\phi$ be the component of $\nabla\phi$ in $\Sym^2_0 TM^*$. From~\cite[p.~507]{uwe}, we have: 
\begin{align}
    \nabla^*\nabla\phi &= \tfrac12 \dd^*\dd \phi +\tfrac1n \dd\dd^*\phi + T^*T\phi, \label{eq:weitz1}\\
    2\,K(R,\pi)\phi &= \tfrac12 \dd^*\dd \phi +\tfrac{n-1}{n} \dd\dd^*\phi - T^*T\phi.\label{eq:weitz2}
\end{align}
Thus, if $\Ric\preceq0$ and $\phi$ is dual to a conformal Killing vector field (i.e., $T\phi=0$), then \eqref{eq:weitz2} implies that $\phi$ is harmonic, hence parallel, so the group $\operatorname{Conf}(M^n,\g)$ of conformal diffeomorphisms has dimension $\leq n$. If $\Ric\prec0$, then $\phi\equiv0$, so $\operatorname{Conf}(M^n,\g)$ is finite.
The latter conclusions also apply to the isometry group $\operatorname{Iso}(M^n,\g)$, since it is a subgroup of $\operatorname{Conf}(M^n,\g)$. These facts are closely related to Theorem~\ref{thm:bochner} applied to 
$TM$ and $TM^*$: it follows from Proposition~\ref{prop:PWforSOn} that $PW_{\SO(n)}(TM)=PW_{\SO(n)}(TM^*)=n-1$, and $\Ric\succ0$ if $R$ is $(n-1)$-positive.

Applying Theorem~\ref{thm:bochner} to $\wedge^p TM^*$, since $t > 0$ for the Hodge Laplacian and $PW_{\SO(n)}(\wedge^p TM^*)=PW_{\SO(n)}(\wedge^{n-p} TM^*)=\max\{p,n-p\}$, as computed in \eqref{eq:pw-wedgep}, one immediately recovers:

\begin{theorem}[Petersen--Wink~{\cite[Thm.~A, B]{petersen-wink}}]\label{thm:pw}
Given $1\leq p \leq \floor{\frac{n}{2}}$, let $q$ be such that $1\leq \min\{q,n-q\}\leq p$, and let $\phi\in\wedge^q TM^*$ be a harmonic $q$-form.
    \begin{enumerate}[\rm (i)]
     \item If $R$ is $(n-p)$-nonnegative, then $\nabla\phi\equiv 0$. In particular, $b_q(M)\leq \binom{n}{q}$;
     \item If $R$ is $(n-p)$-positive, then $\phi\equiv0$. In particular,
    \[b_1(M)=\dots=b_{p}(M)=0, \quad \text{ and }\quad b_{n-p}(M)=\dots =b_{n-1}(M)=0.\]
    \end{enumerate}    
\end{theorem}

Applying Theorem~\ref{thm:bochner} to $\Sym^2_b(\wedge^2 TM)$, we obtain slight generalizations of the Tachibana-type result \cite[Thm.~D]{petersen-wink}, where the \emph{Einstein} assumption is relaxed to \emph{harmonic curvature operator} or \emph{harmonic Weyl tensor} (this was independently observed in~\cite{cmr,petersen-wink-3}). Recall the second Bianchi identity yields~$\dd R=0$, viewing the curvature operator $R$ as a $2$-form with values in $\wedge^2 TM$, so $R$ is called \emph{harmonic} if it is divergence-free, i.e., $\dd^*R=0$, see \cite[Chap.~16]{besse}. Thus, $R$ is harmonic if and only if $0=\Delta_\pi R=\nabla^*\nabla R + 2 K(R,\pi)R$, where $\pi=\Sym^2_b(\wedge^2\R^n)$.

\begin{theorem}\label{thm:pwtachibana}
Suppose $(M^n,\g)$ has harmonic curvature operator, $n\geq5$.
\begin{enumerate}[\rm (i)]
    \item If $R$ is $\frac{n-1}{2}$-nonnegative, then $(M^n,\g)$ is locally symmetric;
    \item If $R$ is $\frac{n-1}{2}$-positive, then $(M^n,\g)$ has constant sectional curvature.
\end{enumerate}
For $n=3$ or $4$, the above statements remain true if $\frac{n-1}{2}$ is replaced with $\tfrac{n}{2}$.
\end{theorem}

\begin{proof}
Set $\G=\SO(n)$. Recall the decomposition of the space of curvature operators $\Sym^2_b(\wedge^2\R^n)=\mathcal U\oplus\mathcal L\oplus\mathcal W$ into $\mathsf O(n)$-irreducibles, where $\mathcal U$, $\mathcal L$, and $\mathcal W$ correspond respectively to multiples of the identity (curvature operators with constant sectional curvature), to the traceless Ricci part, and to the Weyl part.
If $n\geq 5$, then $\mathcal U_\C$, $\mathcal L_\C$, and $\mathcal W_\C$ are $\G$-irreducible and of real type, and respectively isomorphic to the trivial representation, $\pi_{2\varepsilon_1}$, and $\pi_{2\varepsilon_1+2\varepsilon_2}$. 
The same is true if $n=3$ or $4$, except for the fact that $\mathcal W_\C$ is trivial if $n=3$, and it  splits further as $\pi_{\varepsilon_1+\varepsilon_2}\oplus \pi_{\varepsilon_1-\varepsilon_2}$ if $n=4$. Thus, the statement follows from Theorem~\ref{thm:bochner} with $t>0$, and the following computations using Proposition~\ref{prop:PWforSOn} for $n\geq5$, and Definition~\ref{def:PW} otherwise:
\begin{align*}
    PW_{\G}\big(\!\Sym^2_b(\wedge^2\R^n)\big)&=\begin{cases}
         \min\{PW_{\G}(\pi_{2\varepsilon_1}),PW_{\G}(\pi_{2\varepsilon_1+2\varepsilon_2}) \},&\text{if } n\geq 5,\\
          \min\{PW_{\G}(\pi_{2\varepsilon_1}),PW_{\G}(\pi_{\varepsilon_1+\varepsilon_2}),PW_{\G}(\pi_{\varepsilon_1-\varepsilon_2}) \},&\text{if } n=4,\\
          PW_{\G}(\pi_{2\varepsilon_1}),&\text{if } n=3,
    \end{cases}\\
    &=\begin{cases}
        \tfrac{n-1}{2}, & \text{ if } n\geq 5,\\
        2, & \text{ if } n=4,\\
        \frac32, & \text{ if } n=3.
    \end{cases}
    \end{align*}
    Note that $R\in (E_\pi)_0$, where $\pi$ is the $\G$-representation on $E=\Sym^2_b(\wedge^2\R^n)$, if and only if $R\in \mathcal U$; 
    and recall $R$ is parallel if and only if $(M^n,\g)$ is locally symmetric.
\end{proof}

\begin{remark}
If $n\geq4$, then $(M^n,\g)$ has harmonic curvature operator if and only if it has constant scalar curvature and harmonic (i.e., divergence-free) Weyl tensor \cite[\S 16.4]{besse}.
Analyzing the irreducible components of $\nabla R\in TM^*\otimes \Sym^2_b(\wedge^2TM)$ or arguing as \cite{petersen-wink-3} implies Theorem~\ref{thm:pwtachibana} for $(M^n,\g)$ with harmonic Weyl tensor.
\end{remark}

Next, consider $\Sym^p_0 TM$, $p\geq2$. Recall that a section $\phi$ of $\Sym^p_0 TM$
is called a \emph{trace-free conformal Killing tensor} if the projection $P_1(\phi)$ of $\nabla \phi\in TM\otimes \Sym^p_0 TM$ onto $\Sym^{p+1}_0 TM$ vanishes identically.
Following \cite[Prop.~6.6]{hms}, we
write $\nabla\phi=P_1(\phi)+P_2(\phi)+P_3(\phi)$, where $P_i(\phi)$, $i=2,3$, are other first-order differential operators given by projecting $\nabla\phi$ onto the remaining irreducibles. In terms of the above, the Weitzenb\"ock formula $\Delta_\pi=\nabla^*\nabla + t\,K(R,\pi)$ can be rewritten as
\begin{equation*}
	t \, K(R,\Sym^p_0\R^n) = -p\, P_1^* P_1 + (n+p-2)\,P_2^*P_2+P_3^*P_3,
\end{equation*}
in particular, the right-hand side is nonnegative on the space of trace-free conformal Killing tensors. Recall that $t < 0$ for this bundle and $PW_{\SO(n)}(\Sym^p_0 TM)=\frac{n+p-2}{p}$ by \eqref{eq:pw-symp}. Thus, applying Corollary~\ref{pointwise}, we have:

\begin{theorem}\label{thm:killing}
Let $\phi\in\Sym^p_0 TM$ be a trace-free conformal Killing tensor. 
    \begin{enumerate}[\rm (i)]
     \item If $R$ is $\frac{n+p-2}{p}$-nonpositive, then $\nabla\phi\equiv 0$.
     \item If $R$ is $\frac{n+p-2}{p}$-negative, then $\phi\equiv0$.
    \end{enumerate}
\end{theorem}

Theorem~\ref{thm:killing} should be compared with \cite[Thm.~1.6]{DS10} and \cite[Prop.~6.6]{hms}, where (i) and (ii) are proven assuming instead that $\sec\leq0$ and $\sec<0$, respectively.

Finally, setting $\G=\U(m)\subset\SO(2m)$ in Theorem~\ref{thm:bochner}, we also recover:

\begin{theorem}[Petersen--Wink~{\cite[Thm.~B, C]{petersen-wink-cplx}}]\label{thm:pw-cplx}
Let $(M^{2m},\g)$ be a closed K\"ahler manifold and consider its K\"ahler curvature operator $R|_{\mathfrak u(m)}\colon \mathfrak u(m)\to\mathfrak u(m)$.
Given $1\leq p,q \leq m$, let $\phi\in\wedge^{p,q} TM^*$ be a harmonic $(p,q)$-form and set \[C^{p,q}=m+1-\tfrac{p^2+q^2}{p+q}.\]
\begin{enumerate}[\rm (i)]
     \item If $R|_{\mathfrak u(m)}$ is $C^{p,q}$-nonnegative, then $\nabla\phi\equiv 0$. In particular, $h^{p,q}(M)\leq \binom{m}{p}\binom{m}{q}$;
     \item If $R|_{\mathfrak u(m)}$ is $C^{p,q}$-positive, then $\phi\equiv0$. In particular, $h^{p,q}(M)=0$.
    \end{enumerate}    
\end{theorem}

In order to compute $PW_{\U(m)}(\wedge^{p,q} \C^m)=C^{p,q}$, the main step is to obtain the decomposition of $\wedge^{p,q} \C^m$ into $\U(m)$-irreducibles, which can be found in \cite[Sec.~2]{petersen-wink-cplx}, see also \cite{chern-decomp,fujiki}. Namely, denote by $\omega\in\wedge^{1,1}\C^m$ the K\"ahler form, and note that all its exterior powers $\omega^k$ are fixed by $\U(m)$.
Set
$V^{p,q}_k := \wedge^{p-k,0} \C^m \otimes_\C \operatorname{span}\{\omega^k\}\otimes_\C \wedge^{0,q-k} \C^m$, for $0\leq k\leq \min\{p,q\}$. According to \cite[Thm.~2.1]{petersen-wink-cplx}, the modules $\wedge^{p,q}_k \C^m:= V^{p,q}_k\cap (V^{p,q}_{k+1})^\perp$ are $\U(m)$-irreducible and 
\begin{equation*}
\wedge^{p,q} \C^m=\textstyle\bigoplus\limits_{k=0}^{\min\{p,q\}} \wedge^{p,q}_k \C^m,
\end{equation*}
cf.~\cite[Prop.~2.2]{fujiki}. According to \cite[Lemma~2.5]{petersen-wink-cplx}, the highest weight of $\wedge^{p,q}_k \C^m$ is 
\( \lambda(p,q,k)= \varepsilon_1+\dots +\varepsilon_{p-k}-(\varepsilon_{m-(q-k)+1}+\dots+\varepsilon_{m})\). So Definition~\ref{def:PW} and \eqref{eq:rhoUm} yield:
\begin{equation*}
    PW_{\U(m)}(\pi_{\lambda(p,q,k)})=1+\tfrac{m(p+q)-p^2-q^2  -2 k (m-p-q)-2 k^2}{p+q-2 k}.
\end{equation*}
The above is increasing with $k$, so its minimum is achieved with $k=0$, i.e.,
\begin{align*}
    PW_{\U(m)}(\wedge^{p,q} \C^m) &=\min_{0\leq k\leq \min\{p,q\}} PW_{\U(m)}(\pi_{\lambda(p,q,k)}) =PW_{\U(m)}(\pi_{\lambda(p,q,0)}) =C^{p,q}.
\end{align*}
Given the above, Theorem~\ref{thm:pw-cplx} now follows from Theorem~\ref{thm:bochner} applied to $\wedge^{p,q}TM$, keeping in mind that $t>0$ for the Hodge Laplacian on $(p,q)$-forms.

\begin{remark}\label{rem:maxPW-SOn}
In light of Theorem~\ref{thm:bochner}, it is natural to ask which $\G$-representations $\pi$, with $\G=\SO(n)$ or $\Spin(n)$, \emph{maximize} $PW(\pi)$. 
From \eqref{eq:PWforSOn}, with $m=\floor{\frac{n}{2}}$, 
\begin{equation*}
\textstyle\big(n-1 -PW(\pi)\big)(a_1^2+\dots a_m^2)=\sum\limits_{j=1}^m (n-2)a_j^2-(n-2j)a_j\geq \sum\limits_{j=1}^m (n-2j)(a_j^2-a_j)\geq 0,
\end{equation*}
since $a_j^2\geq a_j$ because $a_j\in\Z$.
Thus, $PW(\pi)\leq n-1$, and equality is achieved if and only if $\pi$ is the irreducible representation of highest weight $\omega_1$, $\omega_{m-1}$, or $\omega_m$.
\end{remark}

\section{Twisted spinors}
\label{sec:twistedDirac}

In this section, we prove a general vanishing theorem (Theorem~\ref{GeneralVanishing}) for twisted spinors on closed Riemannian spin manifolds using the representation-theoretic approach to the Bochner technique discussed in Section~\ref{sec:bochner}, namely Proposition~\ref{lowerbound}.
As a consequence of this general result, we prove Theorem~\ref{mainthm:A} in the Introduction.

\subsection{Representation theory of twisted spinors}\label{sec: spinor-rep}
Twisted Dirac operators act on vector bundles $S\otimes E_\pi$ constructed from the representation $\pi_S\otimes \pi$ as explained in Section~\ref{sec:weitzenbock}, where $\pi_S$ is the spinor representation (see Example~\ref{ex:reptheorySO2m}).

An essential step to compute the Petersen--Wink invariant of $\pi_S\otimes \pi$ is to decompose it as a sum of irreducible representations.
Consider, e.g., the tensor product $\pi_S\otimes\pi_{\omega_1}$
of $\pi_S\cong S^+\oplus S^-$ and the defining representation on $\C^n$, where $n=2m$ and $m\geq4$ is even.
Recall from Example~\ref{ex:reptheorySO2m} that $S^+$ and $S^-$ have highest weight $\omega_m$ and $\omega_{m-1}$, respectively.
By Pieri's formula, their tensor products decompose as
$\pi_{\omega_m}\otimes\pi_{\omega_1} \cong  \pi_{\omega_m+\omega_1}\oplus \pi_{\omega_{m-1}}$ and
$\pi_{\omega_{m-1}}\otimes\pi_{\omega_1} \cong  \pi_{\omega_{m-1}+\omega_1}\oplus \pi_{\omega_{m}}$, and hence
\begin{equation}\label{eq:PWSotimesomega1}
\begin{aligned}
PW(\pi_{\omega_m}\otimes\pi_{\omega_1})&=\min\!\big\{PW(\pi_{\omega_m+\omega_1}), PW(\pi_{\omega_{m-1}})\big\},\\
PW(\pi_{\omega_{m-1}}\otimes\pi_{\omega_1})&=\min\!\big\{PW(\pi_{\omega_{m-1}+\omega_1}), PW(\pi_{\omega_{m}})\big\},
\end{aligned}
\end{equation}
according to Definition~\ref{def:PW}. It follows from Proposition~\ref{prop:PWforSOn} that 
\begin{equation}\label{eq:PWpiomegamomega1}
\begin{aligned}
PW(\pi_{\omega_m+\omega_1})=PW(\pi_{\omega_{m-1}+\omega_1})&=\tfrac{n(n+7)}{n+16},\\ PW(\pi_{\omega_{m}})=PW(\pi_{\omega_{m-1}})&= n-1,
\end{aligned}
\end{equation}
and hence the minima in \eqref{eq:PWSotimesomega1} are achieved by the former, so
\begin{equation}\label{eq:PWpiSpiomega1}
PW(\pi_S\otimes\pi_{\omega_1})=\tfrac{n(n+7)}{n+16}.
\end{equation}

In what follows, we perform a similar computation for tensor products of $\pi_S$ with general alternating and symmetric powers of the defining representation $\pi_{\omega_1}\cong\C^n$. Recall from Example~\ref{ex:reptheorySO2m} the decompositions \eqref{eq:decomp-wedgep} and \eqref{eq:decomp-symp} of $\wedge^p\C^n$ and $\Sym^p\C^n$ into irreducible $\SO(n)$-representations, where $n=2m$ is even. For simplicity, we analyze only the components involving tensor products of $S^+\cong \pi_{\omega_m}$, as the computations for those involving $S^-\cong \pi_{\omega_{m-1}}$ are completely analogous and ultimately not needed. We use the convention that $\omega_0=0$, and $\pi_0$ is the trivial representation.

\begin{proposition}\label{prop:PWsympwedgep}
If $n=4k$, $m=2k$, and $k\geq2$, then
\begin{align}
	PW\big( \pi_{\omega_m}\otimes \wedge^p\pi_{\omega_1}\big) &=
	\tfrac{n^2+(8 p-1)n-8 p(p-1)}{n+16 p},
	\qquad 0\leq p\leq m,\label{eq:PWwedge}\\[3pt]
    PW\big( \pi_{\omega_m}\otimes \Sym^p\pi_{\omega_1}\big) &= 
	\tfrac{n^2+(8 p-1)n+8 p (p-1)}{n+8 p (p+1)},
    \qquad p\geq0.	\label{eq:PWsym}
\end{align}
\end{proposition}

\begin{proof}
The decomposition of $\pi_{\omega_m}\otimes \wedge^p\pi_{\omega_1}$ into irreducibles follows from \eqref{eq:decomp-wedgep} and the following instances of the Littlewood--Richardson rule (see~\cite{littelmann}):
\begin{align}
    \pi_{\omega_m}\otimes \pi_{\omega_p} &\cong \textstyle\bigoplus\limits_{j=0}^p \pi_{\omega_{p-j} + \omega_{m- (j - 2\floor{ j/2}) }}, \qquad 1\leq p \leq m-2,\label{eq:decomp-omega_m,p} \\
    \pi_{\omega_m}\otimes \pi_{\omega_{m-1}+\omega_m} &\cong\textstyle \pi_{\omega_{m-1}+2\omega_m} \oplus \bigoplus\limits_{j=1}^{m-1} \pi_{\omega_{m-j-1}+ \omega_{m- (j - 2\floor{ j/2})}},\label{eq:decomp-omega_m,m-1,m} \\
    \pi_{\omega_m}\otimes \pi_{2\omega_{m}} &\cong \pi_{3\omega_{m}} \textstyle\oplus \bigoplus\limits_{j=1}^{m/2} \pi_{\omega_{m-2j} + \omega_{m}},\label{eq:decomp-omega_m,m}\\
    \pi_{\omega_m}\otimes \pi_{2\omega_{m-1}} &\cong \pi_{2\omega_{m-1}+\omega_{m}} \oplus \textstyle\bigoplus\limits_{j=1}^{m/2-1} \pi_{\omega_{m-2j-1} + \omega_{m-1}},\label{eq:decomp-omega_m,m-1}
\end{align}
cf.~\cite[Thm $\mathcal O(2n)$]{kempf-ness}.

By \eqref{eq:PWforSOn}, the Petersen--Wink invariant of the above irreducibles is given by
\begin{align}
        PW(\pi_{\omega_p+ \omega_m})&=PW(\pi_{ \omega_p+\omega_{m-1}})= \tfrac{2 m^2+ (8 p-1)m-4 p(p-1)}{m+8 p},\label{eq:pw-2}\\ 
        PW(\pi_{q\omega_{m-1}+\ell\omega_m}) &= \tfrac{2 m^2 (q+\ell)+m (q+\ell-2) (q+\ell)-4 q \ell}{m (q+\ell)^2-4 q \ell},\label{eq:pw-3} 
\end{align}
where $1\leq p\leq m-2$, $q,\ell\geq0$, and $(q,\ell)\neq(0,0)$.
We proceed case-by-case, following \eqref{eq:decomp-wedgep}.
First, if $1\leq p\leq m-2$, then $\wedge^p\pi_{\omega_1}\cong\pi_{\omega_p}$, so by \eqref{eq:decomp-omega_m,p} and \eqref{eq:pw-2},
\begin{equation}\label{eq:pw-sp-pform}
\begin{aligned}
PW\big( \pi_{\omega_m}\otimes \wedge^p\pi_{\omega_1}\big) &= \min_{0\leq j\leq p} PW(\pi_{\omega_{p-j}+\omega_{m- (j - 2\floor{ j/2})}} )\\
&=\min_{ 0\leq j\leq p }  
\tfrac{2 m^2+ (8 p-8j-1)m-4 (p-j-1) (p-j)}{m+8 (p-j)} \\
&=\tfrac{2 m^2+ (8 p-1)m-4 (p-1) p}{m+8 p}=\tfrac{n^2+(8 p-1)n-8p (p-1)}{n+16 p},
\end{aligned}
\end{equation}
as the minimum above is achieved at $j=0$. 

Second, if $p=m-1$, then $\wedge^{p}\pi_{\omega_1}\cong\pi_{\omega_{m-1}+\omega_m}$, so using \eqref{eq:decomp-omega_m,m-1,m}, \eqref{eq:pw-2}, and \eqref{eq:pw-3},
\begin{multline*}
PW\big( \pi_{\omega_m}\otimes \wedge^{p}\pi_{\omega_1}\big)=\\
 = \min\left\{ PW(\pi_{\omega_{m-1}+2\omega_m}), \min_{1\leq j\leq m-1} PW(\pi_{\omega_{m-j-1}+ \omega_{m- (j - 2\floor{ j/2})} })\right\}\\
=\min\left\{ \tfrac{6 m^2+3 m-8}{9 m-8}, \, \min_{1\leq j\leq m-1} \tfrac{3 m (2 m+1)-4 j (j+3)-8}{9 m-8 j-8} \right\}=\tfrac{6 m^2+3 m-8}{9 m-8}.
\end{multline*}
Finally, if $p=m$, then $\wedge^p\pi_{\omega_1}\cong\pi_{2\omega_{m}}\oplus\pi_{2\omega_{m-1}}$, so using \eqref{eq:decomp-omega_m,m} -- \eqref{eq:pw-3},
\begin{multline*}
PW\big( \pi_{\omega_m}\otimes \wedge^{p}\pi_{\omega_1}\big)=\\
 = \min\left\{ PW(\pi_{3\omega_{m}}),  PW(\pi_{2\omega_{m-1}+\omega_m}), \min_{2\leq j\leq m} PW(\pi_{\omega_{m-j}+ \omega_{m }} )\right\}\\
=\min\left\{ \tfrac{2 m+1}{3}, \, \tfrac{6 m^2+3 m-8}{9 m-8}, \, \min_{2\leq j\leq m}\tfrac{3 m (2 m+1)-4 j (j+1)}{9 m-8 j} \right\} =\tfrac{2 m+1}{3}.
\end{multline*}
Note that both of the above coincide with the values assumed by the last line of \eqref{eq:pw-sp-pform} setting $p=m-1$ and $p=m$, respectively, which concludes the proof of~\eqref{eq:PWwedge}.

The decomposition of $\pi_{\omega_m}\otimes \Sym^p\pi_{\omega_1}$ into irreducibles follows from \eqref{eq:decomp-symp} together with $\pi_{\omega_m}\otimes \pi_{q\omega_1} \cong  \pi_{q \omega_1+\omega_m}\oplus \pi_{(q-1)\omega_1+\omega_{m-1}}$, $q\geq1$, which is a simple consequence of the Littlewood--Richardson rule; 
namely,
\begin{equation}\label{eq:decomp-omega_m,symp}
\begin{aligned}
	 \pi_{\omega_m}\otimes \Sym^p\pi_{\omega_1}&\cong \textstyle\bigoplus\limits_{j=0}^{\floor{p/2}}\big( \pi_{\omega_m} \otimes \pi_{(p-2j)\omega_1}\big)\\
	&\cong\textstyle\bigoplus\limits_{j=0}^{\floor{p/2}} \pi_{(p-2j)\omega_1+\omega_m}\oplus \bigoplus\limits_{j=0}^{\floor{p/2}-1} \pi_{(p-2j-1)\omega_1+\omega_{m-1}}.
\end{aligned}
\end{equation}
By \eqref{eq:PWforSOn}, the Petersen--Wink invariant of these irreducibles is given by
\begin{equation}\label{eq:pw-1}
	PW(\pi_{q\omega_1+\omega_m})=PW(\pi_{q\omega_1+\omega_{m-1}})= \tfrac{2 m^2+(8 q-1) m+4 q(q-1)}{m+4 q (q+1)}, \; q\geq0.
\end{equation}
Thus, by \eqref{eq:decomp-omega_m,symp} and \eqref{eq:pw-1},
\begin{multline*}
PW\big( \pi_{\omega_m}\otimes \Sym^p\pi_{\omega_1}\big)= \min_{0\leq q\leq p}  PW(\pi_{q\omega_1+\omega_m})=\min_{0\leq q\leq p}\tfrac{2 m^2+(8 q-1) m+4 q(q-1)}{m+4 q (q+1)} \\
=\tfrac{2 m^2+ (8 p-1) m+4 p(p-1)}{m+4 p (p+1)}=\tfrac{n^2+(8 p-1)n+8 p (p-1)}{n+8 p (p+1)},
\end{multline*}
since the above minimum is achieved at $q=p$, concluding the proof of \eqref{eq:PWsym}
\end{proof}

\begin{remark}\label{rem:comparePWs}
From Proposition~\ref{prop:PWsympwedgep}, $PW\big( \pi_{\omega_m}\otimes \wedge^p\pi_{\omega_1}\big)\geq PW\big(\pi_{\omega_m}\otimes \Sym^p\pi_{\omega_1}\big)$ for all $0\leq p\leq m$, with equality if $p=0$ or $p=1$,
and strict inequality if $2\leq p\leq m$.
\end{remark}

Note that the constants defined in \eqref{eq:apbp} are precisely \eqref{eq:PWsym} and \eqref{eq:pw-symp}, namely:
\begin{equation}\label{eq:apbp-new}
r_p=PW\big(\pi_{\omega_m}\otimes \Sym^p\pi_{\omega_1}\big) \quad \text{ and }\quad r'_p= PW\big(\Sym^p\pi_{\omega_1}\big).
\end{equation}

\subsection{General vanishing theorem}
Let $(M^n,\g)$ be a closed Riemannian spin manifold with curvature operator $R\colon \wedge^2 TM\to\wedge^2 TM$, and $\G$ be a connected compact Lie subgroup of $\Spin(n)$ that contains $\operatorname{Hol}(M^n,\g)$, as in Section~\ref{sec:weitzenbock}.
Let $S\to M$ be the complex spinor bundle, \(E_\pi\to M\) be another complex vector bundle defined by a unitary $\G$-representation \(\pi\), and consider the bundle \(S\otimes E_\pi\) defined by the tensor product $\G$-representation \(\pi_S\otimes\pi\).
Using Clifford multiplication, the linearization of \(\pi_S\) can be written in terms of an orthonormal basis \(e_i\) of \(\R^{n}\), see \cite[Cor.~I.6.3]{lm-book}:
\begin{equation}\label{sderiv}
\dd\pi_S(e_i\wedge e_j)=\tfrac{1}{2}e_ie_j.
\end{equation}
The twisted Dirac operator \(D_\pi=D_{E_\pi}\) acts on sections of \(S\otimes E_\pi,\) and satisfies $D_{\pi}^2=\nabla^*\nabla+\mathcal{R_\pi}$, where 
\begin{equation}\label{eq:R-HitchinDirac}
\mathcal{R_\pi}=\textstyle\sum\limits_{i<j}(e_ie_j\otimes 1)\circ R^{S\otimes E_\pi}_{e_i,e_j},
\end{equation}
see \cite[p.~155]{lm-book}. Here, \(R^{S\otimes E_\pi}\) is the curvature tensor defined by the connection on \(S\otimes E_\pi\) induced by the representation, hence \(R^{S\otimes E}_{e_i,e_j}=-\dd(\pi_S\otimes \pi) (R(e_i\wedge e_j)).\)

\begin{lemma}\label{HitchinDirac}
The curvature term \eqref{eq:R-HitchinDirac} can be written in terms of \eqref{eq:krrho} as follows:
\begin{equation*}
    \mathcal{R}_\pi=K(R,\pi_{S}\otimes \pi)+\tfrac{\scal}{8}\operatorname{Id}-1\otimes K(R,\pi)
\end{equation*}
\end{lemma}

\begin{proof}
Using \eqref{sderiv} and expanding the linearized tensor product of representations,
\begin{align*}
\mathcal{R}_\pi&=-2\textstyle\sum\limits_{i<j}\left(\dd\pi_S(e_i\wedge e_j)\otimes 1\right)\circ \dd(\pi_S\otimes\pi)(R(e_i\wedge e_j))\\
&=-2\textstyle\sum\limits_{i<j}\left(\dd\pi_S(e_i\wedge e_j)\otimes 1\right)\circ \left(\dd\pi_S(R(e_i\wedge e_j))\otimes 1+1\otimes \dd\pi(R(e_i\wedge e_j))\right)\\
&=-\textstyle\sum\limits_{i<j}\left(\dd\pi_S(e_i\wedge e_j)\otimes 1\right)\circ \left(\dd\pi_S(R(e_i\wedge e_j))\otimes 1+1\otimes \dd\pi(R(e_i\wedge e_j))\right)\\
&\qquad+K(R,\pi_S)\otimes 1-\textstyle\sum\limits_{i<j}\dd\pi_S(e_i\wedge e_j)\otimes \dd\pi(R(e_i\wedge e_j)).
\end{align*}
Since \(R\) is symmetric, one sees that 
\[\textstyle\sum\limits_{i<j}\dd\pi_S(e_i\wedge e_j)\otimes \dd\pi(R(e_i\wedge e_j))=\sum\limits_{i<j}\dd\pi_S(R(e_i\wedge e_j))\otimes \dd\pi(e_i\wedge e_j).\]
Adding and subtracting \(1\otimes K(R,\pi)\) in the second and third lines below, respectively,
\begin{align*}
\mathcal{R}_\pi&=-\textstyle\sum\limits_{i<j}\left(\dd\pi_S(e_i\wedge e_j)\otimes 1\right)\circ \left(\dd\pi_S(R(e_i\wedge e_j))\otimes 1+1\otimes \dd\pi(R(e_i\wedge e_j))\right)\\
&\qquad -\textstyle\sum\limits_{i<j}(1\otimes \dd\pi(e_i\wedge e_j))\circ(\dd\pi_S(R(e_i\wedge e_j))\otimes1+1\otimes \dd\pi(R(e_i\wedge e_j)))\\
&\qquad +K(R,\pi_S)\otimes1-1\otimes K(R,\pi)\\
&=K(R,\pi_{S}\otimes \pi)+K(R,\pi_S)\otimes1-1\otimes K(R,\pi).
\end{align*}
To conclude, recall \(K(R,\pi_{S})=\frac{\scal}{8}\operatorname{Id}\), see  Example~\ref{ex:vanillaDirac} or \cite[Thm.~II.8.8]{lm-book}.
\end{proof}

We now prove a general vanishing theorem for twisted spinors (Theorem~\ref{GeneralVanishing}) which is the primary ingredient in the proof of Theorem~\ref{mainthm:A}. Furthermore, if twisting with \(\Sym^p TM_\C\) or \(\wedge^pTM_\C\), the hypotheses below are \emph{weaker} than those stated in Theorem~\ref{mainthm:A}, since one may
obtain a curvature expression $c(R)\geq C_p(R)$ directly using the computations in the proof of Proposition~\ref{prop:PWsympwedgep} together with \eqref{eq:pw-wedgep} or \eqref{eq:pw-symp}.
  
\begin{theorem}\label{GeneralVanishing}
    Let \((M^{n},\g)\) be a closed Riemannian spin manifold of dimension \linebreak
    $n=4k$, \(k\geq 2\), with curvature operator $R$, let \(\G\) be a connected compact Lie subgroup of $\Spin(n)$ containing $\mathrm{Hol}(M,\g)$, and \(\pi\) be an irreducible unitary \(\G\)-representation with highest weight $\lambda\in P_{++}(\G_\C)$. 
Consider the decomposition \(\pi_S\otimes \pi =\bigoplus_i\pi_{\lambda_i}\) into irreducible $\G$-representations, and let $c(R)\colon M\to\R$ be given by
    \begin{equation}\label{eq: dirac-vanishing-condition}
    c(R)=\min_i\big\{\|\lambda_{i}\|^2\,\Sigma({PW_\G(\pi_{\lambda_i})},R)\big\}+\tfrac{\scal}{8}+\|\lambda\|^2\,\Sigma(PW_\G(\pi),-R),
    \end{equation}
    with $\|\lambda_{i}\|^2\,\Sigma({PW_\G(\pi_{\lambda_i})},R)=0$ if $\lambda_i=0$. 
    If $c(R)\geq0$ and \(E\subseteq E_\pi\) is a parallel complex subbundle, then \(\hat{A}(M,E)=0\), or $c(R)\equiv0$ and \(S\otimes E\) has a nontrivial parallel~section.
\end{theorem}

\begin{proof}
Using the splitting \(S=S^+\oplus S^-\), and the fact that \(E\subseteq E_\pi\) is parallel, 
we can restrict \(D_{\pi}\) to \(S^\pm\otimes E\) to obtain
\(D_{E}^\pm\colon S^\pm\otimes E\to S^\mp\otimes E.\)
By the Atiyah--Singer Index Theorem, see e.g.~\cite[Thm.~13.10]{lm-book}, 
\[\operatorname{ind}(D_{E}^+)=\dim\ker D_{E}^+ - \dim \ker D_{E}^-=\hat{A}(M,E).\] 
If the above is nonzero, then \(\ker{D_{E}}\subset\ker{D_\pi}\) contains a section \(\phi\not\equiv0.\) Combining Lemma~\ref{HitchinDirac} and Proposition~\ref{lowerbound}, the hypotheses imply that \(\mathcal{R}_{\pi}\geq c(R)\operatorname{Id}\). The standard Bochner argument, cf.~\eqref{eq:bochner}, implies that \(\nabla\phi\equiv0\) and \(\mathcal{R}_{\pi}\) has nontrivial kernel at all points; in particular, \(c(R)\equiv 0\).
\end{proof}

\subsection{On the proof of Theorem \ref{mainthm:A}}
First, we address the case \(p=1\), combining the proof of Theorem~\ref{GeneralVanishing} with an analysis of the representations involved if \(\pi=\pi_{\omega_1}\).  

\begin{theorem}\label{general-T}
	Let \((M^{n},\g)\) be a closed Riemannian spin manifold of dimension $n=4k$, \(k\geq2\), with curvature operator \(R\), and $\Ric\preceq \mu\operatorname{Id}$. Set $r_1=\tfrac{n(n+7)}{n+16}$, and 
	\begin{equation*} 
	C_1(R)=\min\left\{\left(\tfrac{n}{8}+2\right)\Sigma({r_1},R),\tfrac\scal8\right\} +\tfrac{\scal}{8} -\mu.
	\end{equation*}
	If $C_1(R)\geq0$ and \(E\subseteq TM_\C\) is a parallel complex subbundle, then \(\hat{A}(M,E)=0\), or else $C_1(R)\equiv0$ and \(S\otimes E\) has a nontrivial parallel section.
\end{theorem}

\begin{proof}
We argue as in Theorem~\ref{GeneralVanishing}, setting \(\pi=\pi_{\omega_1}\) to be the lift to $\G=\Spin(n)$ of the defining representation of \(\SO(n)\), so that $E_\pi=TM$. 
Letting $m=2k$, then
\[\pi_S\otimes\pi_{\omega_1}\cong \pi_{\omega_m+\omega_1}\oplus\pi_{\omega_{m-1}}\oplus \pi_{\omega_{m-1}+\omega_1}\oplus\pi_{\omega_{m}},\]
as explained in Section~\ref{sec: spinor-rep}.
Using Proposition~\ref{prop:basicsKRpi} (iii) and Example~\ref{ex:vanillaDirac}, we have
\begin{align*}
K(R,\pi_{S}\otimes\pi_{\omega_1})&=K(R,\pi_{\omega_m+\omega_1})\oplus K(R,\pi_{\omega_{m-1}})\oplus K(R,\pi_{\omega_{m-1}+\omega_1})\oplus K(R,\pi_{\omega_{m}})\\
&=K(R,\pi_{\omega_m+\omega_1})\oplus\left(\tfrac\scal 8\operatorname{Id}\right)\oplus K(R,\pi_{\omega_{m-1}+\omega_1})\oplus\left(\tfrac\scal 8\operatorname{Id}\right).
\end{align*}
Applying Proposition~\ref{lowerbound}, since 
$PW(\pi_{\omega_m+\omega_1})=PW(\pi_{\omega_{m-1}+\omega_1})=r_1$ by \eqref{eq:PWpiomegamomega1}, and $\|\omega_m+\omega_1\|^2=\|\omega_{m-1}+\omega_1\|^2=\frac{n}{8}+2$,
we obtain
\begin{equation*}
	K(R,\pi_{\omega_m+\omega_1})\succeq \left(\tfrac{n}{8}+2\right)\Sigma({r_1},R) \operatorname{Id}, \; \text{and} \; K(R,\pi_{\omega_{m-1}+\omega_1})\succeq \left(\tfrac{n}{8}+2\right)\Sigma({r_1},R) \operatorname{Id}.
\end{equation*}
Thus, $K(R,\pi_{S}\otimes\pi_{\omega_1})\succeq\min\left\{\left(\tfrac{n}{8}+2\right)\Sigma({r_1},R),\tfrac\scal8\right\}\operatorname{Id}$. 
From Example~\ref{ex:TM}, we have \(K(R,\pi_{\omega_1})=\Ric\).
Therefore, by Lemma~\ref{HitchinDirac}, we have \(\mathcal{R}_{\pi}\succeq C_1(R)\operatorname{Id}\succeq0\) and the conclusion follows as in the proof of Theorem~\ref{GeneralVanishing}.
\end{proof}

\begin{remark}
Despite similarities,  $C_1(R)$ does not coincide with $c(R)$ in \eqref{eq: dirac-vanishing-condition} for $\pi=\pi_{\omega_1}$, since we explicitly compute $K(R,\pi_{\omega_1})$, $K(R,\pi_{\omega_m})$ and $K(R,\pi_{\omega_{m-1}})$ in the proof above using Examples~\ref{ex:vanillaDirac} and \ref{ex:TM} instead of appealing to Proposition~\ref{lowerbound}. 
\end{remark}

Second, in order to address the case \(p\geq 2\), we need compare the relevant quantities in Theorem~\ref{GeneralVanishing} for different subrepresentations of the \(p\)-tensor representation $\pi_{\omega_1}^{\otimes p}$, and of their tensor products with the spinor representation $\pi_S\cong\pi_{\omega_m}\oplus\pi_{\omega_{m-1}}$.

\begin{lemma}\label{lem:PW-ineq}
Let $n=4k$, $k\geq2$. Recall the constants $r_p$ and $r'_p$ in \eqref{eq:apbp}, cf.\ \eqref{eq:apbp-new}, and the Weyl vector $\rho$ in \eqref{eq:rho}.
If \(\pi\) is an  nontrivial irreducible subrepresentation of the \(\SO(n)\)-representation \(\pi_{\omega_1}^{\otimes p}\), $p\geq 2$, with highest weight $\lambda\in P_{++}(\SO(n,\C))$, 
	\begin{equation*}
	PW(\pi)\geq r'_p, \quad \text{ and }\quad 0\leq\left<\lambda,\lambda+2\rho\right>\leq p(n+p-2).
	\end{equation*}
	Denoting also by $\pi$ its lift to $\Spin(n)$, let $\pi_S\otimes\pi=\bigoplus_i \pi_{\lambda_i}$ be the decomposition into irreducible components. Then
	\begin{equation*}
	PW(\pi_{\lambda_i})\geq r_p,\quad \text{ and }\quad \tfrac{n(n-1)}{8}\leq\left<\lambda_i,\lambda_i+2\rho\right>\leq
	p(n+p-1)+\tfrac{n(n-1)}{8}.
	\end{equation*}
\end{lemma}

\begin{proof}
Let $m=2k$. Since \(\lambda \in P_{++}(\SO(n,\C))\),
we have that \(\lambda=\sum_{j=1}^m a_j\varepsilon_j\) with \(a_1\geq a_2\geq\dots\geq|a_m|\geq0\) and $a_j\in\Z$ for all $1\leq j\leq m$, see Example~\ref{ex:reptheorySO2m}.
By Proposition~\ref{prop:PWforSOn}, the Petersen--Wink invariant \(PW(\pi)\) does not depend on the sign of \(a_m\) because $n=2m$ is even, see \eqref{eq:PWforSOn}, so we may assume \(a_m=|a_m|\geq0\).
Moreover, 
\(\lambda\) is a weight of \(\pi_{\omega_1}^{\otimes p}\), so it can be written as \(\lambda=\pm\varepsilon_{j_1}\pm\dots\pm\varepsilon_{j_p}\) and thus \(\sum_{j=1}^m a_j\leq p\).
Altogether, the vector $(a_1,\dots,a_m)$ of coefficients belongs to
\[
\triangle_p=\left\{(a_1,\dots,a_m)\in\R^m : a_1\geq a_2\geq \dots \geq a_{m}\geq 0\;\text{ and }\; \textstyle\sum\limits_{j=1}^m a_j\leq p\right\},\]
which is an $m$-simplex in $\R^m$; namely $\triangle_p$ is the convex hull of the vertices $v_0=0$ and \(v_q=\big(\frac{p}{q},\dots,\frac{p}{q},0,\dots,0\big)\in\R^q\oplus\{0\}\subseteq\R^m\), \(1\leq q\leq m\).
From \eqref{eq:rho}, we have
\begin{equation}\label{eq:concavefct}
p\left<\lambda,2\rho\right>-2(m-1)\|\lambda\|^2=2\textstyle\sum\limits_{j=1}^m\left(p (m-j)\, a_j -(m-1)\, a_j^2\right).
\end{equation}
We claim that the right-hand side of \eqref{eq:concavefct} is nonnegative for all $(a_1,\dots,a_m)\in\triangle_p$. Indeed, this is a concave function of $(a_1,\dots,a_m)$ and hence attains its minimum on $\triangle_p$ at a vertex $v_q$, $0\leq q\leq m$. It can be easily checked that this minimum is equal to zero and it is achieved at the vertices $v_0$ and $v_1$.
Therefore, as desired,
\[PW(\pi)=\tfrac{\left<\lambda,\lambda+2\rho\right>}{\|\lambda\|^2}\geq1+ \tfrac{2(m-1)}{p}= \tfrac{n+p-2}{p} =r'_p.\]
Moreover, using \eqref{eq:rho} again, we obtain
\begin{equation*}
0\leq\left<\lambda,\lambda+2\rho\right>\leq\left(\textstyle\sum\limits_{j=1}^m a_j \right)^2+2(m-1)\textstyle\sum\limits_{j=1}^m a_j \leq  p^2+2p(m-1)=p(n+p-2).
\end{equation*}

Let us also denote by $\pi$ the lift to $\Spin(n)$ of the $\SO(n)$-representation $\pi$, and recall that the weights of $\pi_S$ are of the form \(\sum_{i=1}^m\pm\frac12\varepsilon_i\).  Thus, the highest weight \(\lambda_i\in P_{++}(\mathfrak{so}(n,\C))\) of an irreducible component of $\pi_S\otimes \pi$ must be of the form \(\lambda_i=\sum_{j=1}^m b_j \varepsilon_j\), with \(b_j=a_j\pm\frac12\) and \(a_j\in\Z\) such that \(\sum_{j=1}^m|a_j|\leq p\). It follows that \(|b_m|=|a_m\pm\frac12|\geq\frac12\); in particular, $\lambda_i\neq 0$. Again, in light of \eqref{eq:PWforSOn}, the sign of \(b_m\) is irrelevant to compute $PW(\pi_{\lambda_i})$ so we may assume \(b_m=|b_m|\geq 0\). Altogether, the vector $(b_1,\dots, b_m)$ of coefficients belongs to the (translated) simplex
\[\triangle_p+\big(\tfrac12,\dots,\tfrac12\big)=\left\{(b_1,\dots,b_m)\in\R^m : b_1\geq \dots \geq b_{m}\geq \tfrac12\;\text{ and }\; \textstyle\sum\limits_{j=1}^m b_j\leq p+\frac{m}{2} \right\}\]
whose vertices are $v_q+\big(\tfrac12,\dots,\tfrac12\big)$, $0\leq q\leq m$. Similarly to \eqref{eq:concavefct}, we have that
\begin{multline*}
(\tfrac{m}{4}+p^2+p)\left<\lambda_i,2\rho\right>-2(m-1)(\tfrac{m}{4}+p)\|\lambda_i\|^2=\\
2\textstyle\sum\limits_{j=1}^m (\tfrac{m}{4}+p^2+p)(m-j)\,b_j -(m-1)(\frac{m}{4}+p)\,b_j^2
\end{multline*}
is nonnegative for all $(b_1,\dots,b_m)\in\triangle_p+\big(\tfrac12,\dots,\tfrac12\big)$, since this is a concave function of $(b_1,\dots,b_m)$ and its minimum on $\triangle_p+\big(\tfrac12,\dots,\tfrac12\big)$ is equal to zero and achieved at the vertex $v_1+\big(\tfrac12,\dots,\tfrac12\big)$. Therefore, as desired,
\begin{equation*}
	PW(\pi_{\lambda_i}) = \tfrac{\left<\lambda_i,\lambda_i+2\rho\right>}{\|\lambda_i\|^2}\geq 1+\textstyle\frac{2(m-1)(\frac{m}{4}+p)}{\tfrac{m}{4}+p^2+p} =\frac{n^2+(8p-1)n+8p(p-1)}{n+8p(p+1)}= r_p.
\end{equation*}
Finally, we may bound
\[\left<\lambda_i,\lambda_i+2\rho\right>=\textstyle\sum\limits_{j=1}^m \left(a_j^2\pm a_j+\tfrac14\right)+ 2 \sum\limits_{j=1}^m \left(a_j\pm\frac{1}{2}\right)(m-j)\]	
from above with
\[\left<\lambda_i,\lambda_i+2\rho\right>\leq \tfrac{m}{4}+p^2+p+2p(m-1)+\tfrac{(m-1)m}{2}=	p(n+p-1)+\tfrac{n(n-1)}{8},\]
and from below with
\[\left<\lambda_i,\lambda_i+2\rho\right>\geq\textstyle\sum\limits_{j=1}^m \tfrac14+\sum\limits_{j=1}^m(m-j)=\frac{m(2m-1)}{4}=\tfrac{n(n-1)}{8},\]
where we use that $a_j^2\pm a_j\geq0$ because $a_j\in\mathds Z$, and $b_j=a_j\pm\tfrac12\geq\tfrac12$ because $(b_1,\dots,b_m)\in\triangle_p+\big(\tfrac12,\dots,\tfrac12\big)$.
\end{proof}

Using Lemma~\ref{lem:PW-ineq}, we shall now apply Theorem~\ref{GeneralVanishing} to prove the following result, which implies Theorem~\ref{mainthm:A} in the remaining case \(p\geq 2\).   Note that a parallel subbundle \(E_\C\subseteq TM_\C^{\otimes p}\) decomposes into the direct sum of parallel bundles, each contained in \(E_\pi\) for some irreducible subrepresentation \(\pi\) of \(\pi_{\omega_1}^{\otimes p}\).  

\begin{theorem}\label{thm: ptensor-vanishing}
	Let \((M^{n},\g)\) be a closed Riemannian spin manifold of dimension \linebreak $n=4k$, \(k\geq2\), with curvature operator \(R\). Let \(\pi\) be an irreducible subrepresentation of the \(\SO(n)\)-representation \(\pi_{\omega_1}^{\otimes p}\), \(p\geq 2\), with highest weight $\lambda\in P_{++}(\SO(n,\C))$,~and 
	\[C_p(R)=\textstyle \min\!\Big\{\!\left(\frac{n}{8}+p^2+p\right)\Sigma({r_p},R),\,\frac{n(n-1)}{8r_p}\,\Sigma({r_p},R) \!\Big\}+\frac{\scal}{8}+p^2\,\Sigma({r'_p},-R),\]
	with $r_p$ and $r'_p$ as in \eqref{eq:apbp}. If $C_p(R)\geq 0$ and \(E\subseteq E_\pi\) is a parallel complex subbundle, then \(\hat{A}(M,E)=0\), or $C_p(R)\equiv0$ and $S\otimes E$ has a nontrivial parallel~section.
\end{theorem}

\begin{proof}
	Let \(\pi_{S} \otimes\pi =\bigoplus_i\pi_{\lambda_i}\) be the decomposition into irreducible components. Recalling that \(r\mapsto \Sigma(r,R)/r\) is nondecreasing in \(r\), since $PW(\pi_{\lambda_i})\geq r_p$ by Lemma~\ref{lem:PW-ineq}, 
\[\|\lambda_i\|^2\,\Sigma(PW(\pi_{\lambda_i}),R)=\left<\lambda_i,\lambda_i+2\rho\right>\frac{\Sigma({PW(\pi_{\lambda_i})},R)}{PW(\pi_{\lambda_i})} \geq\left<\lambda_i,\lambda_i+2\rho\right> \frac{\Sigma(r_p,R)}{r_p}.
\]
Using Lemma~\ref{lem:PW-ineq} once more, if \(\Sigma(r_p,R)\leq 0\), then
	\begin{multline}\label{eq:ineq1}
	\|\lambda_i\|^2\,\Sigma(PW(\pi_{\lambda_i}),R)\geq
\left(p(n+p-1)+\tfrac{n(n-1)}{8}\right)
	\frac{\Sigma(r_p,R)}{r_p}\\ =\left(\tfrac{n}{8}+p^2+p \right)\Sigma(r_p,R)
	\end{multline}
	otherwise, if \(\Sigma(r_p,R)\geq 0\), then
		\begin{equation}\label{eq:ineq2}
		\|\lambda_i\|^2\,\Sigma(PW(\pi_{\lambda_i}),R)
	\geq\frac{n(n-1)}{8}\frac{\Sigma(r_p,R)}{r_p}.
	\end{equation}
	
We claim that \(\Sigma(r'_p,-R) \leq 0\). Indeed, suppose by contradiction \(\Sigma(r'_p,-R) > 0.\) Then \(2\,\Sigma\big({\binom{n}{2}},-R\big)=-\scal > 0\) and \(\Sigma(r_p,R) < 0\), so
\begin{align*}
C_p(R)& < \left(\tfrac{n}{8}+p^2+p \right)\Sigma(r_p,R)+p^2\,\Sigma(r'_p,-R)\\
&=\left(p(n+p-1)+\tfrac{n(n-1)}{8}\right)\frac{\Sigma(r_p,R)}{r_p}+p(n+p-2)\frac{\Sigma(r'_p,-R)}{r'_p}.
\end{align*}
Since 
\(\frac{\Sigma(r_p,R)}{r_p}\leq-\frac{\Sigma(r'_p,-R)}{r'_p}\leq 0\) and \( p(n+p-1)+\tfrac{n(n-1)}{8}  > p(n+p-2)\), we conclude that \(C_p(R)< 0,\) which contradicts our hypotheses, proving that \(\Sigma(r'_p,-R)\leq 0\).

From Lemma~\ref{lem:PW-ineq}, we have $PW(\pi)\geq r'_p$ and as $r\mapsto\Sigma(r,R)/r$ is nondecreasing,
\begin{multline}\label{eq:ineq3}
\|\lambda\|^2\, \Sigma({PW(\pi)},-R) =\left<\lambda,\lambda+2\rho\right>\tfrac{\Sigma({PW(\pi)},-R)}{PW(\pi)}\\
\geq \left<\lambda,\lambda+2\rho\right>\tfrac{\Sigma(r'_p,-R)}{r'_p} \geq p(n+p-2) \tfrac{\Sigma(r'_p,-R)}{r'_p}= p^2\,\Sigma(r'_p,-R),
\end{multline}
where the last inequality uses Lemma~\ref{lem:PW-ineq} once again. 

Therefore, combining \eqref{eq:ineq1}, \eqref{eq:ineq2}, and \eqref{eq:ineq3}, we conclude that, for all $i$,
	\[\|\lambda_{i}\|^2\, \Sigma({PW(\pi_{\lambda_i})},R)+\tfrac{\scal}{8}+\|\lambda\|^2\,\Sigma({PW(\pi)},-R)\geq C_p(R),\]
	and hence the conclusion follows from Theorem~\ref{GeneralVanishing}. 
\end{proof}

\subsection{\texorpdfstring{Monotonicity of $C_p(R)$}{Monotonicity}}
Using similar arguments, we now show that the curvature conditions \(C_p(R)>0\) are nested, and each implies \(\scal>0\); that is:

\begin{proposition}\label{prop:nested}
	Let \(n\geq 3\) and \(1\leq q <p\). If $R\in\Sym^2_b(\wedge^2\R^n)$ has \(C_p(R)\geq0\), then \(\frac{\scal}{4}\geq C_q(R)\geq C_p(R).\) 
\end{proposition}

\begin{proof}
	First, we note that \[\frac{\scal}{8}=\frac{n(n-1)}{8}\frac{\Sigma\left(\binom{n}{2},R\right)}{\binom{n}{2}}\geq\frac{n(n-1)}{8}\frac{\Sigma(r_1,R)}{r_1},\]
	and that \(\mu\), being a trace of \(R\) over a subspace of \(\wedge^2\R^n\) of dimension \(r'_1=n-1\), satisfies \(\mu\leq-\Sigma(r'_1,-R)\). Thus, it follows that 
	\[C_1(R)\geq\textstyle \min\Big\{\!\left(\frac{n}{8}+2\right) \Sigma({r_1},R),\, \frac{n(n-1)}{8r_1}\,\Sigma({r_1},R) \Big\}+\frac{\scal}{8}+\Sigma({r'_1},-R),\]
	and the right-hand side is the result of setting $p=1$ on the formula for \(C_p(R)\),~\(p\geq 2\).
	
	Since \(n>2\), both \(r_p\) and \(r_p'\) are decreasing functions of \(p\).  As demonstrated in the proof of Theorem~\ref{thm: ptensor-vanishing}, it follows from \(C_p(R)>0\) that \(\Sigma(r'_p,-R)\leq 0.\) Thus,
	\begin{multline*}
    q^2\,\Sigma(r'_q,-R)=q(n+q-2) \frac{\Sigma(r'_q,-R)}{r'_q}\geq q(n+q-2) \frac{\Sigma(r'_p,-R)}{r'_p}\\\geq p(n+p-2) \frac{\Sigma(r'_p,-R)}{r'_p}=p^2\,\Sigma(r'_p,-R).
    \end{multline*}
	If \(\Sigma(r_q,R)\leq 0\), then 
	\begin{multline*}
    \textstyle \min\Big\{\!\left(\frac{n}{8}+q^2+q\right) \Sigma({r_q},R),\, \frac{n(n-1)}{8r_q}\, \Sigma({r_q},R)\Big\}=\left(\frac{n}{8}+q^2+q\right)\!\Sigma({r_q},R)\\
    =\tfrac{1}{8}\left(n^2+(8q-1)n+8q(q-1)\right)\,\Sigma(r_q,R)/r_q\\ 
    \geq\tfrac{1}{8}\left(n^2+(8p-1)n+8p(p-1)\right)\,\Sigma(r_q,R)/r_q\\
    \geq\tfrac{1}{8}\left(n^2+(8p-1)n+8p(p-1)\right)\,\Sigma(r_p,R)/r_p\\
    =\left(\tfrac{n}{8}+p^2+p\right) \Sigma({r_p},R)\geq \textstyle \min\Big\{\!\left(\tfrac{n}{8}+p^2+p\right)\Sigma({r_p},R),\, \frac{n(n-1)}{8r_p}\, \Sigma({r_p},R) \!\Big\},
    \end{multline*}
	while, if \(\Sigma(r_q,R)\geq 0\), then 
	\begin{multline*}
    \textstyle \min\Big\{\!\left(\frac{n}{8}+q^2+q\right)\Sigma({r_q},R),\,\frac{n(n-1)}{8r_q}\,\Sigma({r_q},R) \Big\}=\frac{n(n-1)}{8r_q}\Sigma({r_q},R)\\
	\geq \textstyle\frac{n(n-1)}{8r_p} \Sigma({r_p},R)\geq \textstyle \min \Big\{\!\left(\frac{n}{8}+p^2+p\right) \Sigma({r_p},R),\, \frac{n(n-1)}{8r_p}\, \Sigma({r_p},R) \Big\}.
    \end{multline*}
	Therefore, in all cases, \(C_q(R)\geq C_p(R)\).
	
	Finally, assume \(C_1(R)\geq0.\) If \(\mu< 0,\) then \(\scal<0,\) and since \(r_1<r_0=n-1\) and \(\mu\geq \Sigma(n-1,R)\), we have
	\[C_1(R)\leq\tfrac{1}{8}{n(n+7)}\frac{\Sigma(r_1,R)}{r_1}-\mu\leq\tfrac{1}{8}{n(n+7)}\frac{\Sigma(n-1,R)}{n-1}-(n-1)\frac{\Sigma(n-1,R)}{n-1}\]\[=\left(\tfrac{n^2-n+8}{8n-8}\right)\Sigma(n-1,R)<0.\]
	Since this contradicts our assumption \(C_1(R)\geq0\), we conclude that \(\mu\geq0,\) and 
	\[C_1(R)\leq\textstyle \min\Big\{\!\left(\frac{n}{8}+2\right) \Sigma({r_1},R),\,\frac{\scal}{8}\Big\}+\frac{\scal}{8}\leq\frac{\scal}{4}.\qedhere\]
	\end{proof}

\section{Cobordism classes}\label{cobordism}

If \((M^n,\g)\) is a closed Riemannian manifold of dimension $n=4k$, $k\geq2$, whose curvature operator is $r$-positive with $2k \leq r\leq n-1$, then Theorem~\ref{thm:pw} implies the vanishing of its Betti numbers \(b_1,\dots ,b_{n-r}\), and \(b_r, \dots ,b_{n-1}\). (To simplify notation, throughout this section, all Betti numbers \(b_i=b_i(M)\) and Pontryagin numbers $p_{I}=p_{I}(M)$ are understood to refer to \(M\), and all rational Pontryagin classes \(p_i=p_i(TM)\) to \(TM\).) 
Thus, the rational Pontryagin classes in the corresponding degrees vanish, as do any Pontryagin numbers involving those Pontryagin classes.  If, in addition, the conditions in Theorem~\ref{mainthm:A} are satisfied, then further linear combinations of Pontryagin numbers vanish. In this section, we combine these results to give sufficient conditions for \emph{all} Pontryagin numbers to vanish, that is, for \(M\) to be rationally null-cobordant. We first prove Theorem~\ref{mainthm:D} (ii) and (iii), as follows:

\begin{theorem}\label{general-cobordism}
Let \((M^{n},\g)\) be a closed Riemannian spin manifold of dimension $n=4k$, with \(k\geq 6\) and \(k\neq 7\). If its curvature operator is \(r\)-positive, where
\begin{equation*}
    r= 2k+4 \text{ if }k\text{ is even}, \quad \text{ and }\quad r= 2k+6 \text{ if }k\text{ is odd},
\end{equation*}
and \(\frac{\scal}{8}\operatorname{Id}-\Ric\succeq0\), then \(M\) is rationally null-cobordant. 
\end{theorem}

\begin{proof}
First, suppose \(k=2\ell\) is even. 
By Theorem~\ref{thm:pw}, since $(M^n,\g)$ has $r$-positive curvature operator with $r=4\ell+4$, the only possibly nonvanishing Betti numbers of \(M\) besides \(b_0\) and \(b_n\) are \(b_{4\ell-3}, \dots, b_{4\ell+3}\).  Thus, all rational Pontryagin classes vanish, except possibly \(p_\ell\) and \(p_{2\ell}\), i.e., the only possibly nonvanishing Pontryagin numbers are \(p_{(\ell,\ell)}\) and \(p_{(2\ell)}\). We now prove that these also must vanish.

In the above setup, a direct computation (e.g., using tools in \cite[\S 1.8]{HBJ}) gives
\[\hat{A}(TM)=1+c_{2\ell}\, p_\ell+\tfrac12\left(c_{2\ell}^2-c_{4\ell}\right)p_\ell^2+c_{4\ell}\, p_{2\ell},\]
where \(c_i\) is the coefficient of \(x^i\) in the power series expansion of \(\frac{x}{4\tan(x/2)}\) at $x=0$. 
Since $M$ is spin and has \(\scal>0\), by Lichnerowicz, we have
\begin{equation}\label{eq:lineqn1}
	\hat{A}(M)=\tfrac12\left(c_{2\ell}^2-c_{4\ell}\right)p_{(\ell,\ell)}+c_{4\ell}\,p_{(2\ell)}=0.
\end{equation}
Using Newton's identities for power sums and elementary symmetric polynomials,
\[\operatorname{ch}(TM_\C)=8\ell+\tfrac{(-1)^{\ell+1}}{(2\ell-1)!}\,p_\ell+\tfrac{2\ell}{(4\ell)!}\,p_\ell^2-\tfrac1{(4\ell-1)!}\,p_{2\ell}.\]
The curvature operator $R$ of $(M^n,\g)$ is $r$-positive with \(r<r_1\) and \(\frac\scal 8\operatorname{Id}-\Ric\succeq0\), therefore \(C_1(R)>0\). Thus, we may apply Theorem~\ref{mainthm:A} with $p=1$ and obtain
\begin{equation}\label{eq:lineqn2}
	\hat{A}(M,TM_\C)= \left(\tfrac{(-1)^{\ell+1}}{(2\ell-1)!}c_{2\ell}+\tfrac{2\ell}{(4\ell)!}\right)p_{(\ell,\ell)} -\tfrac1{(4\ell-1)!}\,p_{(2\ell)}=0.
\end{equation}
In order to show that \(p_{(\ell,\ell)}\) and \(p_{(2\ell)}\) vanish, it suffices to show that the homogeneous linear system given by \eqref{eq:lineqn1} and \eqref{eq:lineqn2} on those variables only has the trivial solution. This is easily seen to be equivalent to
\[c_{2\ell}\neq0\quad\text{ and }\quad (-1)^{\ell}(2\ell-1)!\,{c_{2\ell}}\neq2(4\ell-1)!\,c_{4\ell}.\] 
Using that \(2(-1)^\ell(2\ell)!\,c_{2\ell}=B_{2\ell}\) is the \(2\ell^\text{th}\) Bernoulli number, the above conditions are satisfied if and only if \(B_{2\ell}\neq0\), which always holds, and \(B_{2\ell}\neq B_{4\ell}\), which holds if \(\ell\neq 2\), and we assumed $k=2\ell\geq6$. Thus, \(p_{(\ell,\ell)}=p_{(2\ell)}=0\), as desired.

Next, assume \(k=2\ell+1\) is odd. By Theorem~\ref{thm:pw},
since the curvature operator $R$ of $(M^n,\g)$ is $r$-positive with $r=4\ell+8$, aside from \(b_0\) and \(b_{n}\), all Betti numbers vanish except possibly \(b_{4\ell-3}, \dots, b_{4\ell+7}\). Thus, the only possibly nonzero rational Pontryagin classes are \(p_{\ell}\), \(p_{\ell+1}\), and \(p_{2\ell+1}\), i.e., the only possibly nonzero Pontryagin numbers are $p_{(\ell,\ell+1)}$ and $p_{(2\ell+1)}$.
Similarly to the above, in this situation, as \(\ell>1\), 
\begin{align*}
\hat{A}(TM)&=1+c_{2\ell}\,p_\ell+c_{2\ell+2}\,p_{\ell+1}+\left(c_{2\ell}c_{2\ell+2}-c_{4\ell+2}\right)\,p_{\ell}\,p_{\ell+1}+c_{4\ell+2}\,p_{2\ell+1}, \\
\operatorname{ch}(TM_\C)&=8\ell+4+\tfrac{(-1)^{\ell+1}}{(2\ell-1)!}\,p_\ell+\tfrac{(-1)^\ell}{(2\ell+1)!}\,p_{\ell+1}-\tfrac1{(4\ell+1)!}\,p_\ell\,p_{\ell+1}+\tfrac1{(4\ell+1)!}\,p_{2\ell+1}.
\end{align*}
Once again, $\hat A(M)=0$ because $M$ is spin and has \(\scal>0\), 
and \(\hat{A}(M,TM_\C)=0\) by Theorem~\ref{mainthm:A}, because \(C_1(R)>0\), as $R$ is $r$-positive with $r<r_1$ and \(\frac\scal 8\operatorname{Id}-\Ric\succeq0\).
The homogeneous linear system given by \(\hat{A}(M)=\hat{A}(M,TM_\C)=0\) 
on the variables \(p_{(\ell,\ell+1)}\) and \(p_{(2\ell+1)}\) only has the trivial solution provided that
\[-(4\ell+2)\,B_{2\ell}\,B_{2\ell+2}+(2\ell+2)\,B_{2\ell}\,B_{4\ell+2}+(2\ell)\,B_{2\ell+2}\,B_{4\ell+2}\neq0,\]
which holds if \(\ell\neq 2,3\), and we assumed \(\ell\geq4\). 
Thus \(p_{(\ell,\ell+1)}=p_{(2\ell+1)}=0\). 
\end{proof}

Let us now address the remaining statement (i) in Theorem~\ref{mainthm:D}, regarding the case $k=2$. In dimension $n=8$, since $\frac{\scal}{8}\operatorname{Id}-\Ric$ is traceless, the condition \(\frac{\scal}{8}\operatorname{Id}-\Ric\succeq0\) is equivalent to the Einstein condition $\Ric=\frac{\scal}{8}\operatorname{Id}$.
Repeating the  proof of Theorem~\ref{general-cobordism} with \(k=2\) and \(r=5\), it follows that if \((M^8,\g)\) is a closed Riemannian spin manifold with an Einstein metric and \(5\)-positive curvature operator, then all its Pontryagin numbers vanish. In this dimension, the cobordism group \(\Omega_{8}^\Spin\cong\Z\oplus\Z\) has no torsion and is hence completely determined by Pontryagin numbers, so we conclude that such \(M^8\) is null-cobordant, as claimed.

\smallskip
Next, we consider the relevant dimensions not covered by Theorem~\ref{general-cobordism}.
In these dimensions, under the analogous hypotheses, the homogeneous linear system given by \(\hat{A}(M)=\hat{A}(M,TM_\C)=0\) on the only $2$ possibly nonvanishing Pontryagin numbers \emph{degenerates}, i.e., admits nontrivial solutions. 
Thus, we must ensure the vanishing of some other linear combination of these Pontryagin numbers, and \(\hat{A}(M,\wedge^2TM_\C)\) turns out to be a judicious choice. In order to obtain its vanishing, we make an assumption on \(K(R,\wedge^2TM)=K(R,\wedge^2\pi_{\omega_1})\), which is given explicitly in Example~\ref{ex:Labbi-formulas}; recall also that \(-K(R,\wedge^2TM)=K(-R,\wedge^2TM)\) can be bounded  from below in terms of the $n-2$ largest eigenvalues of \(R\) by Proposition~\ref{lowerbound} and~\eqref{eq:pw-wedgep}.

\begin{theorem}\label{thm:left-overs}
Let \((M^{n},\g)\) be a closed Riemannian spin manifold 
of dimension $n=4k$, \(k=4\), \(5\), or $7$. If its curvature operator is $r$-positive, where \(r=\frac{4 k^2+ 15k-4}{k+8}\), and \(\frac{\scal}{8}\operatorname{Id}-K(R,\wedge^2TM)\succeq0\), then \(M\) is rationally null-cobordant.
\end{theorem}

\begin{proof}
If \(k=4\), then \(r=10\), so \(p_1=0\) and \(p_3=0\) by Theorem~\ref{thm:pw}. Computing as in the proof of Theorem~\ref{general-cobordism} in terms of the only remaining Pontryagin numbers,
\begin{equation}\label{eq:lin1}
    \hat{A}(M)=\tfrac{1}{2^{11}\cdot3^4\cdot5^2\cdot7}\left(13\,p_{(2,2)}-2^2\cdot3\,p_{(4)}\right)=0.
 \end{equation}
Since \(R\) is \(r\)-positive with $r=PW(\pi_{S}\otimes \wedge^2\pi_{\omega_1})$ by Proposition~\ref{prop:PWsympwedgep}, we obtain from Corollary~\ref{pointwise} (ii) that \(K(R,\pi_S\otimes\wedge^2\pi_{\omega_1})\succ0\), and thus that \(\mathcal R_{\pi}\succ0\) for $\pi=\wedge^2\pi_{\omega_1}$ by Lemma~\ref{HitchinDirac}. 
Therefore, \(\hat{A}(M,\wedge^2TM_\C)=0.\)  Using \eqref{eq:ch-weights}, or computing \(\operatorname{ch}(\wedge_tTM)\) with the splitting principle and the multiplicative property of \(\wedge_t\), see \eqref{eq:symwedget}, we have
\begin{equation}\label{eq:lin2}
	\hat{A}(M,\wedge^2TM_\C)=\tfrac{1}{2^8\cdot3\cdot5\cdot7}\left(101\,p_{(2,2)}+2^2\cdot 149\,p_{(4)}\right)=0.
\end{equation}
Since the homogeneous linear system given by \eqref{eq:lin1} and \eqref{eq:lin2} only admits the trivial solution, it follows that $p_{(2,2)}=p_{(4)}=0$, so all  Pontryagin numbers vanish.

If \(k=5\), then \(r<16\), so \(p_1=p_4=0\), and by the same arguments, we have
\begin{align*}
\hat{A}(M)&=\tfrac{1}{{2^{11}\cdot{3^5}\cdot{5^2}\cdot{7}\cdot{11}}}\left(3\cdot7\,p_{(2,3)}-2\cdot5\,p_{(5)}\right)=0,\\
\hat{A}(M,\wedge^2TM_\C)&=\tfrac{1}{{{2^{10}}\cdot{3^5}\cdot{5^2}\cdot{7}\cdot{11}}}\left({{3}\cdot{7}\cdot{23}\cdot{73}}\,p_{(2,3)}-2\cdot 5\cdot 13\cdot 5003\,p_{(5)}\right)=0,
\end{align*}
which only has the trivial solution, so all Pontryagin numbers vanish. 

Finally, if \(k=7\), then \(r<20\), so \(p_1=p_2=p_5=p_6=0\), and similarly we have
\begin{align*}
\hat{A}(M)&=\tfrac{1}{{{2^{15}}\cdot{3^6}\cdot{5^3}\cdot{7^2}\cdot{11}\cdot{13}}}\left(283\,p_{(3,4)}-2^2\cdot 5\cdot 7\,p_{(7)}\right)=0,\\
\hat{A}(M,\wedge^2TM_\C)&=\tfrac{1}{{{2^{14}}\cdot{3^5}\cdot{5^3}\cdot{7^2}\cdot{11}\cdot{13}}}\left(-227\cdot1009\,p_{(3,4)}-{{2^2}\cdot{5}\cdot{7}\cdot{32719}}\,p_{(7)}\right)=0,
\end{align*}
which only has the trivial solution, so all Pontryagin numbers vanish.
\end{proof}

The only dimensions $n=4k$ not addressed in Theorem~\ref{mainthm:D} nor in Theorem~\ref{thm:left-overs} are the cases $k=1$ and $3$. For $k=1$, it follows directly from the vanishing of \(\hat{A}(M)\) that a closed Riemannian spin $4$-manifold with $\scal>0$ is null-cobordant. For $k=3$, Theorem~\ref{thm:pw} implies that a closed Riemannian $12$-manifold with \(8\)-positive curvature operator has \(p_1=p_2=0.\) If such $M^{12}$ is spin, then \(\hat{A}(M)=0\) and thus \(p_3=0,\) so once again $M$ is rationally null-cobordant.

\section{Elliptic genus and Witten genus}\label{witten-genus}

In this section, we use modularity of the elliptic genus \(\varphi\) and of the Witten genus \(\varphi_W\), as defined in Section~\ref{sec: modular forms}, to derive sufficient conditions for their vanishing (Theorem~\ref{mainthm:witten}). First, we prove a  lemma with elementary considerations in the theory of modular forms; recall the definition \eqref{eq: gamma0} of the subgroup \(\Gamma_0(2)\subset \SL(2,\Z)\).

\begin{lemma}\label{ordinf}
    Let \(f\in M_m(\SL(2,\Z))\) be a modular form of weight $m$. If
    \begin{enumerate}[\rm (i)]
        \item \(m\not\equiv 2 \mod 12\) and \(\operatorname{ord}_\infty(f)> \floor{\tfrac{m}{12}}\), or 
        \item \(m\equiv 2 \mod 12\) and \(\operatorname{ord}_\infty(f)>\floor{\tfrac{m}{12}}-1\),
    \end{enumerate} 
then \(f=0.\) If \(g\in M_m(\Gamma_0(2))\) and \(\operatorname{ord}_\infty(g)>\floor{\tfrac{m}{4}}\), then \(g=0.\)  
\end{lemma}

\begin{proof}
    As \(\begin{psmallmatrix}1&1\\0&1\end{psmallmatrix} \in\Gamma_0(2)\), we have \(f(\tau+1)=f(\tau)\) and \(g(\tau+1)=g(\tau)\), so \(f\) and \(g\) have a Fourier expansion in integer powers of \(q=e^{2\pi i\tau}\) and \(\mathrm{ord}_\infty(f),\mathrm{ord}_\infty(g)\in\Z\) are nonnegative integers, i.e., \(N=1\) in the notation of Section~\ref{sec: modular forms}. If \(f\neq0\), then 
    \[\textstyle\sum_{\tau\neq i,\,e^{2\pi i/3}}\mathrm{ord}_\tau(f)+\tfrac{1}{3}\,\mathrm{ord}_{e^{2\pi i/3}}(f)+\tfrac{1}{2}\,\mathrm{ord}_{i}(f)+\mathrm{ord}_\infty(f)=\tfrac{m}{12},\]
    see \cite[Appendix I, Thm~4.1]{HBJ} or \cite[Prop.~2]{zagier}. 
    In case (i), as all terms on the left-hand side are nonnegative, 
    it is strictly larger than the right-hand side yielding the desired contradiction.
    In case (ii), we have \(m\equiv 2 \mod 12\), so if \(f\neq 0\),
    \[4\,\mathrm{ord}_{e^{2\pi i/3}}(f)\equiv 2\mod 6.\]
    Since \(\mathrm{ord}_{e^{2\pi i/3}}(f)\geq 0\), by the above \(\mathrm{ord}_{e^{2\pi i/3}}(f)\geq 2\), so we obtain the contradiction
    \[\mathrm{ord}_\infty(f)\leq \tfrac16+\floor{\tfrac{m}{12}}-\tfrac23<\floor{\tfrac{m}{12}}.\]
    The final statement follows once again by contradiction using the equivalent formula for \(g\in M_m(\Gamma_0(2))\), see \cite[Appendix I, \S 4.2]{HBJ}, namely, if $g\neq 0$ then
    \[\textstyle\sum_{\tau\neq\frac{1+i}{2}}\mathrm{ord}_\tau(g)+\tfrac{1}{2}\,\mathrm{ord}_{\frac{1+i}{2}}(g)+2\,\mathrm{ord}_0(g)+\mathrm{ord}_\infty(g)=\tfrac{m}{4}.\qedhere\]
\end{proof}

We now prove Theorem~\ref{mainthm:witten} using the above lemma together with Theorem~\ref{mainthm:A} to find curvature conditions which imply that sufficiently many coefficients of the Fourier expansion of \(\varphi\) or \(\varphi_W\) vanish so that the entire modular form vanishes.

\begin{proof}[Proof of Theorem~\ref{mainthm:witten}]
	Consider the formal power series of bundles, see \eqref{eq:symwedget}, 
	\[\textstyle\bigotimes\limits_{\ell=1}^\infty\Sym_{q^{\ell}}TM\quad\text{ and }\quad \bigotimes\limits_{\ell=1}^\infty\wedge_{-q^{2\ell-1}}TM\otimes\Sym_{q^{2\ell}}TM.\]
	Each is a product of sums of terms of the form \(Eq^d\), with \(E\subseteq TM^{\otimes s}\) and \(s \leq d\). Multiplication of two such terms preserves that property, so the coefficient of \(q^d\) in each bundle is a (formal) linear combination of parallel subbundles of \(TM^{\otimes s}\) with \(s\leq d\). Such property is also preserved multiplying by a power series with scalar coefficients.
	Thus, the coefficient of \(q^d\) in \(\varphi_W(M)(\tau)\) and in \(\widetilde\varphi(M)(2\tau)\), see \eqref{eq:witten-genus} and \eqref{tildephi} respectively, is a linear combination of terms \(\hat{A}(M,E_\C)\) for parallel subbundles \(E\subseteq TM^{\otimes s}\), \(s\leq d\).

	To prove (ii), if \(C_{\floor{k/2}}(R)>0\), then all such terms \(\hat{A}(M,E_\C)\) vanish for \(d\leq \floor{k/2}\) by Theorem~\ref{mainthm:A}, hence \(\operatorname{ord}_\infty(\tilde\varphi(M)(2\tau))>\floor{k/2}\). Since \(\widetilde\varphi(M)(2\tau)\in M_{2k}(\Gamma_0(2))\), as $n=2m=4k$, Lemma~\ref{ordinf} implies that \(\widetilde\varphi(M)(2\tau)=0\), so \(\varphi(M)=0\) by \eqref{eq: phi_phi_tilde}.
	
	Next, to prove (i), assume \(p_1(TM)=0\), so that \(\varphi_W(M)(\tau)\in M_{2k}(\SL(2,\Z))\). If \(C_p(R)>0\), with \(p\) as in the statement of the theorem, then,  as above, Theorem~\ref{mainthm:A} implies that \(\operatorname{ord}_\infty(\varphi_W(M)(\tau))>p\), so \(\varphi_W(M)=0\) by Lemma~\ref{ordinf} with $m=2k$.   
\end{proof}

\section{Examples and surgery stability}\label{section: ex}

In this section, we examine some examples of closed Riemannian manifolds \((M^{n},\g)\) whose curvature operator \(R\) satisfies \(C_p(R)>0\) and prove Theorem~\ref{A-props}.

As a first example, consider the unit round sphere $\Ss^n$, for which $R_{\Ss^n}=\operatorname{Id}$, hence
\begin{equation}\label{eq:CpSphere}
C_p(R_{\Ss^n})= \tfrac14 n^2 -\left(p+ \tfrac14\right)n - p(p-2), \quad p\geq1.
\end{equation}
Recall that the curvature operator $R_{M\times N}$ of a product $M\times N$ is $R_M\oplus R_N\oplus 0$ on $\wedge^2(T(M\times N))\cong\wedge^2 TM\oplus\wedge^2 TN\oplus (TM\otimes TN)$.
 
\begin{proposition}\label{A-products}
	The following hold:
	\begin{enumerate}[\rm (i)]
		\item Let \((M^{n},\g)\) be a Riemannian manifold of dimension $n$ with an Einstein metric  satisfying $R\succeq0$ and $\scal>0$. (For instance, $M$ can be chosen to be a product of compact rank one symmetric spaces \(\Ss^{q}\), $\C P^{q}$, $\mathds{H}P^q$, \(\Ca P^2\).) Let \(H_{i_1 j_1},\dots,H_{i_k j_k}\) be Milnor surfaces, see Example~\ref{ex:generators-cobordism}, and \(N\) be any closed manifold. If \(n+2(j_1-1)+\dots +2(j_k-1)>8\), then the manifold
		\[M \times H_{i_1 j_1}\times\dots \times H_{i_k j_k}\times N\]
		admits a metric with \(C_1(R)>0\).
		\item The condition \(C_1(R)>0\) is stable under surgeries of codimension \(d\geq 10\).
		\item The condition \(C_p(R)>0\), $p\geq 2$, is stable under surgeries of codimension \(d\) on manifolds of dimension $n$ provided that \((d-1)(d-2)>8p(p+n-2)\).
		\end{enumerate}
\end{proposition}
 
\begin{proof}
The Milnor surface \(H_{ij}\) is the total space of a \(\C P^{j-1}\)-bundle over \(\C P^i\). It can be equipped with a connection metric using the two Fubini--Study metrics, and by scaling the metric on the base by a large positive constant the curvature operator can be made arbitrarily close to that of \(\C P^{j-1}\times \R^{2i}\) with the standard product metric. By scaling any metric on \(N\) similarly, the  manifold in (i) admits a metric with curvature operator \(R\) arbitrarily close to the curvature operator $R_*$ of 
\[{M}\times\C P^{j_1-1}\times\dots\times\C P^{j_k-1}\times \R^{2(i_1+\dots+i_k)+\dim N},\]
where the non-Euclidean factors can be scaled such that the metric on their product is Einstein. Since $R_*\succeq0$ and the largest eigenvalue of the Ricci operator of $R_*$ is $\mu=\frac{\scal}{\ell}$, where \(\ell=n+2(j_1-1)+\dots+2(j_k-1)\), we have that $C_1(R_*)\geq  \frac{\scal}{8}-\frac{\scal}{\ell}$. In particular, with $\ell>8$, it follows that \(C_1(R)\) is arbitrarily close to $C_1(R_*)>0$, hence $C_1(R)>0$.

For each $p\geq1$, the set $\big\{R \in \Sym^2_b(\wedge^2\R^n) : C_p(R)>0\big\}$ is an open convex $\mathsf O(n)$-invariant cone, since \(\scal\) and \(\Ric\) are linear in \(R\), and \(\Sigma(r,R)\) is concave in~\(R\). Let $R_d$, $3\leq d\leq n$, be the curvature operator of the product metric on $\R^{n-d+1}\times \Ss^{d-1}$. One easily checks that $r_p<\dim\ker R_d$ for all $p\geq1$, so $C_1(R_d)>0$ if $d\geq10$, and 
\[C_p(R_d)\geq \tfrac18 (d-1)(d-2) -p(p+n-2), \quad \text{ if }p\geq2.\]
Therefore, statements (ii) and (iii) follow directly from Theorem~\ref{hoelzel}.
\end{proof}

\begin{remark}
    For all $n\geq4$ and $p\geq2$, the convex cone of curvature operators $R\in \Sym^2_b(\wedge^2\R^n)$ satisfying $C_p(R)\geq0$ is a \emph{spectrahedron}, as a consequence of \cite[Thm.~3.3]{sanyal2022spectral}. Thus, determining membership in this set is an algebraic task that can be efficiently completed using semidefinite programming, see \cite{siaga}.
\end{remark}

As a consequence of Proposition~\ref{A-products} (i), the curvature condition \(C_1(R)>0\) imposes no restriction on the Betti numbers \(b_1,\dots,b_{n-9}\) of an orientable manifold \(M^{n}\) of dimension \(n\geq10\); in particular, by Poincar\'e duality, no restrictions on \emph{any Betti numbers} if \(n\geq18\). Next, we show that in the absence of the spin condition, \(C_1(R)>0\) imposes \emph{no restriction} on the rational cobordism type nor on Pontryagin numbers in large enough dimensions, while, in the spin setting, the consequences of Theorem~\ref{mainthm:A} are the \emph{only} restriction on rational spin cobordism type:
  
\begin{proposition}\label{A-cobordism}
	The following hold:
	\begin{enumerate}[\rm (i)]
		\item If \([M^{n}]\in\Omega_{n}^\SO\) is not torsion, \(n\geq10\), then \(M\) is oriented cobordant to a connected manifold admitting a metric with \(C_1(R)>0\).
		\item If, furthermore, \(M^{n}\) is spin and \(\hat{A}(M)=\hat{A}(M,TM_\C)=0\), then, for some \(\ell \), the manifold \(\#^\ell M^n\) is spin cobordant to a manifold with \(C_1(R)>0\).    
	\end{enumerate}
\end{proposition}
  
\begin{proof}
By Theorem~\ref{thom-milnor} (ii), the set \(\big\{[\C P^{2m}], \, [H_{ij}] : m\geq1,\, i\geq2,\, j\geq6\big\}\) generates \(\Omega_*^{\SO}/\text{torsion}\). Thus, \([M^{n}]\) can be represented by an  integer linear combination of products to which Proposition~\ref{A-products} (i) applies. Furthermore \(C_1(R)>0\) is preserved by connected sums (surgeries of codimension \(d=n\)) by Proposition~\ref{A-products}~(ii). Therefore, the linear combination can be replaced with a connected sum with the proper orientations, while preserving the curvature condition, which proves (i).
 
Let \(K^4\) represent a generator of \(\Omega_{4}^\Spin\), e.g., a $K3$ surface, so that \(\hat{A}(K^4)=-2\). By Theorem~\ref{thom-milnor} (i), the set \(\{[K^4],[\mathds{H}P^k] : k\geq 2 \}\) generates \(\Omega_*^\Spin\otimes\mathds{Q}.\) Noting that 
 \[\hat{A}\big(M\times N,T(M\times N)_\C\big)=\hat{A}(M)\,\hat{A}(N,TN_\C)+\hat{A}(N)\,\hat{A}(M,TM_\C),\]  
it follows that a product \(N^q\) of elements of the above generating set satisfies \(\hat{A}(N)=\hat{A}(N,TN_\C)=0\) unless it is \((K^4)^q\), for which \(\hat{A}(N)=(-2)^q\), or \((K^4)^{q-2}\times \mathds{H}P^2\), for which \(\hat{A}(N,TN_\C)=-(-2)^{q-2}\). So, if the conditions in (ii) hold, \([M]\) is represented by a rational linear combination of products with factors either \(\mathds{H}P^2\times \mathds{H}P^2\) or \(\mathds{H}P^k\), \(k>2\).  Each such product admits a metric with \(C_1(R)>0\) by Proposition~\ref{A-products} (i). Thus, there is an integer \(\ell \) such that \(\#^\ell M^n\) is spin cobordant to a connected sum (of the above products) admitting a metric with \(C_1(R)>0.\)
 \end{proof}

Theorem~\ref{A-props} follows from Proposition~\ref{A-products} together with Proposition~\ref{A-cobordism}.

\smallskip

Let us now discuss examples of manifolds that admit $C_p(R)>0$, \(p\geq2\). First, by \eqref{eq:CpSphere} and Proposition~\ref{A-products} (iii), the condition \(C_p(R)>0\) is satisfied by round spheres in sufficiently large dimensions, and is stable under connected sums and surgeries of high codimension. Thus, we can construct examples of manifolds with \(C_p(R)>0\) having arbitrarily large first Betti number, as well as other Betti numbers of low degree.
In order to analyze examples which are not null-cobordant, recall the spectrum of the curvature operator of compact rank one symmetric spaces \cite{bk}.

\begin{table}[ht]
\begin{tabular}{|c|l|l|}
\hline
$M$ & Eigenvalues \rule[-2ex]{0pt}{0pt} \rule{0pt}{3ex} & Multiplicity \\
\hline \noalign{\medskip} \hline 
$\C P^m$, $m\geq 2$ & 
\begin{tabular}{l}
	$0$\rule[-1.2ex]{0pt}{0pt} \rule{0pt}{2.5ex} \\
	$2$ \\
	$2m+2$ 
\end{tabular}
 & 
 \begin{tabular}{l}
	$m(m-1)$ \\
	$m^2-1$ \\
	$1$ 
\end{tabular}   \\
\hline
$\mathds H P^{k}$, $k\geq 2$ & 
\begin{tabular}{l}
	$0$ \rule[-1.2ex]{0pt}{0pt} \rule{0pt}{2.5ex} \\
	$4$ \\
	$4k\phantom{+m.}$ 
\end{tabular}
 & 
 \begin{tabular}{l}
	$3(2k+1)(k-1)$ \\
	$k(2k+1)$ \\
	$3$ 
\end{tabular} \\
\hline
$\Ca P^{2}$ & 
\begin{tabular}{l}
	$0\phantom{abc;}$ \rule[-1.2ex]{0pt}{0pt} \rule{0pt}{2.5ex} \\
	$8\phantom{abc;}$
\end{tabular}
 & 
 \begin{tabular}{l}
	$84$ \\
	$36$ 
\end{tabular} \\
\hline
\end{tabular}
\vspace{-.2cm}
\caption{Eigenvalues of the curvature operator $R_M$ of projective spaces $M$ endowed with the Fubini--Study metric with $1\leq\sec\leq 4$.}
\label{tab:eigenv}
\end{table}

All curvature operators $R$ in Table~\ref{tab:eigenv} satisfy $\dim\ker R>r_p$ and $\dim\operatorname{Im} R>r'_p$, hence are such that $C_p(R)=\frac{\scal}{8}+p^2\Sigma({r'_p},-R)$.
Thus, it follows that, for all \(p\geq 2\),
\begin{equation}\label{eq:CpR-crosses}
\begin{aligned}
C_p(R_{\C P^m}) &= \tfrac{1}{2}m^2+\big(\tfrac12-4p-2p^2\big)m-2p(p-2),\\
C_p(R_{\mathds H P^k}) &= \begin{cases}
    2 k^2+4  \left(1-4 p-3 p^2\right) k +8 p (p+1), & \text{ if } p \leq 2k-1,\\
    2 (1-8 p) k^2 + 4\left(1+2 p- p^2\right)k, & \text{ if } p  > 2k-1.
\end{cases}
\end{aligned}
\end{equation}
Clearly, for each fixed $p\geq 2$, the above are positive in sufficiently large dimension; while, for each fixed dimension, they are positive for finitely many $p\geq 2$.
However,
\[C_p(R_{\Ca P^2}) = 72-112 p-8 p^2 \]
is negative for all $p\geq 2$. Note this is in accordance with Theorem~\ref{mainthm:A}, since, e.g., $M=\Ca P^2$ has $\hat A(M,\wedge^2 TM_\C)\neq0$ hence does not admit a metric with $C_2(R)>0$.

\smallskip

Next, consider \(p\) defined in terms of the dimension $n=4k$ as in Theorem~\ref{mainthm:witten} (i): 
\[p=\lfloor{\tfrac{k}{6}}\rfloor-1 \text{ if } k\equiv 1\mod 6, \quad \text{ and } p=\lfloor{\tfrac{k}{6}}\rfloor \text{ otherwise}.\]
For \(p=1\), corresponding to dimensions $n$ between $24$ and $44$ as well as dimension $n=52$, in which the Witten genus $\varphi_W(M)$ is determined by \(\hat{A}(M)\) and \(\hat{A}(M,TM_\C)\), the examples of Propositions~\ref{A-products} and \ref{A-cobordism} demonstrate that a spin manifold with \(p_1(TM)=0\) has vanishing Witten genus if and only if it is rationally cobordant to a spin manifold admitting a metric with \(C_1(R)>0.\)
For \(p=2\), corresponding to dimensions $n$ between $48$ and $68$ as well as dimension $n=76$, one checks with \eqref{eq:CpR-crosses} that the Fubini--Study metric on \(\C P^{32},\C P^{34},\) and  \(\C P^{38},\) along with product metrics on \(\C P^{2}\times\C P^{32},\) \(\C P^{2}\times\C P^{36},\) \(\C P^4\times\C P^{34}\), \(\C P^2\times\C P^2\times\C P^{34}\), and \(\mathds{H}P^2\times\C P^{34}\) 
satisfy \(C_p(R)>0\). In the case of products, an appropriate scaling on certain factors is required. Those examples show that \(C_2(R)>0\) is not so stringent as to imply rational null-cobordism for a general closed oriented manifold. 
 
\providecommand{\bysame}{\leavevmode\hbox to3em{\hrulefill}\thinspace}
\providecommand{\MR}{\relax\ifhmode\unskip\space\fi MR }
\providecommand{\MRhref}[2]{%
  \href{http://www.ams.org/mathscinet-getitem?mr=#1}{#2}
}
\providecommand{\href}[2]{#2}

\end{document}